\documentclass[11pt,oneside,reqno]{amsart}

\usepackage{graphicx}
\usepackage{pict2e}
\usepackage{amssymb}
\usepackage{amsthm}                
\usepackage[margin=1in]{geometry}  
\usepackage{graphics}
\usepackage{epstopdf}
\usepackage{amsmath}
\usepackage{graphicx}
\usepackage{subfigure}
\usepackage{latexsym}
\usepackage{amsmath}
\usepackage{amsfonts}
\usepackage{amssymb}
\usepackage{tikz}
\usepackage{mathdots}
\usepackage{comment}
\usepackage{float}
\usepackage[mathscr]{euscript}
\usepackage{enumerate}
\usepackage{epsfig}
\usepackage { graphicx }
\usepackage { hyperref }
\usepackage { graphics }
\usepackage { graphicx }
\usepackage{xparse}

\usepackage {eufrak}
\usepackage[utf8]{inputenc}
\usepackage{longtable}

\usepackage{amsrefs}

\theoremstyle{definition}
\newtheorem{theorem}{Theorem}[section]
\newtheorem{lemma}[theorem]{Lemma}
\newtheorem{corollary}[theorem]{Corollary}
\newtheorem{proposition}[theorem]{Proposition}

\theoremstyle{definition}
\newtheorem{definition}[theorem]{Definition}
\newtheorem{example}[theorem]{Example}
\theoremstyle{remark}
\newtheorem{remark}[theorem]{Remark}
\numberwithin{equation}{section}

\numberwithin{equation}{section}

\begin{document}

\title{Pretzel Knots and q-Series}

\author{Mohamed Elhamdadi}
\address{Department of Mathematics, University of South Florida, 
Tampa, FL 33647 USA}
\email{emohamed@math.usf.edu }

\author{Mustafa Hajij}
\address{Department of Mathematics, University of South Florida, 
Tampa, FL 33647 USA}
\email{mhajij@usf.edu}




\begin{abstract}
The tail of the colored Jones polynomial of an alternating link is a $q$-series invariant whose first $n$ terms coincide with the first $n$ terms of the $n$-th colored Jones polynomial. Recently, it has been shown that the tail of the colored Jones polynomial of torus knots give rise to Ramanujan type identities. In this paper, we study $q$-series identities coming from the colored Jones polynomial of pretzel knots. We prove a false theta function identity that goes back to Ramanujan and we give a natural generalization of this identity using the tail of the colored Jones polynomial of Pretzel knots. Furthermore, we compute the tail for an infinite family of Pretzel knots and relate it to false theta function-type identities.
\end{abstract}

 \maketitle

 \tableofcontents
\section{Introduction}
The discovery of the Jones polynomial using Von Neumann algebras \cite{J1, J2} and its generalizations \cite{HOMFLY} and \cite{PT} lead to quantum invariants of knots and $3$-manifolds. The Kauffman bracket polynomial \cite{Kauffman} is the simplest interpretation of the Jones polynomial using 
knot diagrams. Reshetikhin and Turaev \cite{RT} gave the first rigorous construction of quantum invariants as linear sums of quantum invariants of framed links.  Soon after, various approaches of constructing quantum invariants were developed using different methods such as using surgery along links \cite{BHMV92, Lic92, TuraevWenzl93} and simplicial complexes \cite{TuraevViro92}.\\
The colored Jones polynomial $J_{n,L}(q)$ of a link $L$ can be understood as a sequence of polynomials with integer coefficients 
that take values in $\mathbb{Z}[q, q^{- 1} ] $. The label $n$ stands for the coloring. The polynomial $J_{2,L}(q)$ is the original Jones polynomial. Recently, there has been a growing interests in the coefficient of the colored Jones polynomial. Dasbach and Lin \cite{DL} used the definition of the colored Jones polynomial coming from Kauffman bracket skein theory to show that for an alternating link $L$ the absolute value of the first and the last three leading coefficients of $J_{n,L}(q)$ are independent of the color $n$, for large values of $n$.  As a consequence, they obtained lower and upper bounds for the volume of the knot complement for an alternating prime non-torus knot $K$ in terms of the leading two and
last two coefficients of $J_{2,K}(q)$ extending their previous result from \cite{DL1}. In \cite{DL} it was conjectured that the first $n$ coefficients of $J_{n,L}(q)$ agree with the first $n$ coefficients of $J_{n+1,L}(q)$ for any alternating link $L$.   This gives rise to a $q$-power series called the tail of the colored Jones polynomial of the alternating link $L$ with many interesting properties. Using skein theory, Armond gave a proof in \cite{Armond}  for the existence of the tail of the colored Jones polynomial of adequate links, hence alternating links and also for closures of positive braids in \cite{Armond1}. Garoufalidis and L\^e \cite{GL} used $R$-matrices to prove the existence of the tail of the colored Jones polynomial of alternating links and proved that higher order stabilization also occur. An alternative proof for the stability was also given in \cite{Hajij3}. In \cite{Hajij1}, the second author investigated certain skein element in the relative Kauffman bracket skein module of the disk with some marked points in order to compute the head and the tail of the colored Jones polynomial obtaining a simple $q$-series for the tail of the knot $8_5$, the first knot in the knot table that is not directly obtained from the work in \cite{CodyOliver}. This investigation was generalized to the study of tail of quantum spin networks in \cite{Hajij2}.
\\
One of the earliest connection between the colored Jones polynomial and Ramanujan type $q$-series was made in \cite{Hikami} in which the author investigated the asymptotic behaviors of the colored Jones polynomials of torus knots. However, the point of view in \cite{Hikami} is different from the point of view of \cite{Hajij1,Hajij2} that we shall adopt here. This point of view allows us to prove more $q$-series identities in a structured manner. Among many interesting properties that the tail of the colored Jones polynomial enjoys as $q$-series is that it is equal to theta functions or false theta functions for many knots with small crossing numbers. For instance all knots in the knots table up to $8_4$, the tail of their colored Jones polynomial are Ramanujan theta, false theta functions or a product of these functions as demonstrated in \cite{CodyOliver}. This does not seem to be the case of knot $8_5$ whose tail is computed in \cite{Hajij1}. More interestingly, the study of the tail has been used to prove  Andrews-Gordon identities for the two variable Ramanujan theta function in \cite{CodyOliver} and a corresponding identities for the false theta function in \cite{Hajij2}. These two families of $q$-series identities were obtained from investigating $(2,p)$-torus knots. For $q$-series techniques proving these identities refer to \cite{Robert}. \\
 In this paper we show that similar observations hold for other natural family of knots, namely Pretzel knots. In particular, we show that pretzel knots give rise to a natural family of $q$-series identities.  
 The paper is organized as follows. In section \ref{sec2} we review the basics of skein theory, some number theory relevant to our work, and some review of the colored Jones polynomial. In section \ref{sec2.5} we list the main results of this paper. Section \ref{sec3} is devoted to Ramanujan type identities that were recovered in the literature using the tail of the colored Jones polynomial and we show how our contribution here fits in this literature. In section \ref{sec4} we give an explicit formula for the tail of colored Jones polynomial of the Pretzel knots $P(2u+1,2,2k+1)$ where $k,u \geq 1$ . In section \ref{sec5} we use two skein theoretic techniques to compute the tail of the colored Jones polynomial of a certain family of pretzel knots and we show that these computations give rise to a Ramanujan type identities. 


\section{Review of Skein Theory and Colored Trivalent Graphs}\label{sec2}
%
%
%
%
%
%
%
%
%
%
%
%
%
\subsection{Skein Theory}
Let $\tilde{\mathbb{Z}}[A,A^{-1}]$ denotes the set of rational functions $\frac{P}{Q}$ where $P,Q \in \mathbb{Z}[A,A^{-1}]$. Let $M$ be an orientable $3$-manifold. A framed link in $M$ is a disjoint
union of oriented annuli embedded into $M.$ Let $\mathcal{L}_M$ be the set of all isotopy classes of framed links in $M$.  We consider the empty link to be an element of $\mathcal{L}_M$. Denote by $\tilde{\mathbb{Z}}[A,A^{-1}] \mathcal{L}_M$ the free $\tilde{\mathbb{Z}}[A,A^{-1}]$-module generated by $\mathcal{L}_M$. Three framed links $L$, $L_{0}$, and $L_{\infty}$ are said to be \textit{Kauffman skein related} if they can be embedded in $M$ identically except in a ball where they appear as in the Figure \ref{kauffmanrelated}
below.

\begin{figure}[H]
  \centering
   {\includegraphics[scale=0.2]{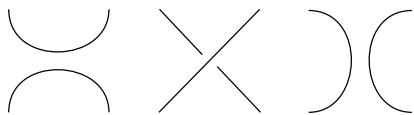}
     \caption{From left to right, $L_{\infty}$, $L$ and $L_{0}$.}
  \label{kauffmanrelated}}
\end{figure}

 If $L$, $L_{0}$ and $L_{\infty}$ are skein related then an expression of the form $L-AL_{0}-A^{-1}L_{\infty}$ is called a \textit{skein relation}. On the other hand, an expression of the form \begin{eqnarray*}\hspace{3 mm} L\sqcup
   \begin{minipage}[h]{0.05\linewidth}
        \vspace{0pt}
        \scalebox{0.02}{\includegraphics{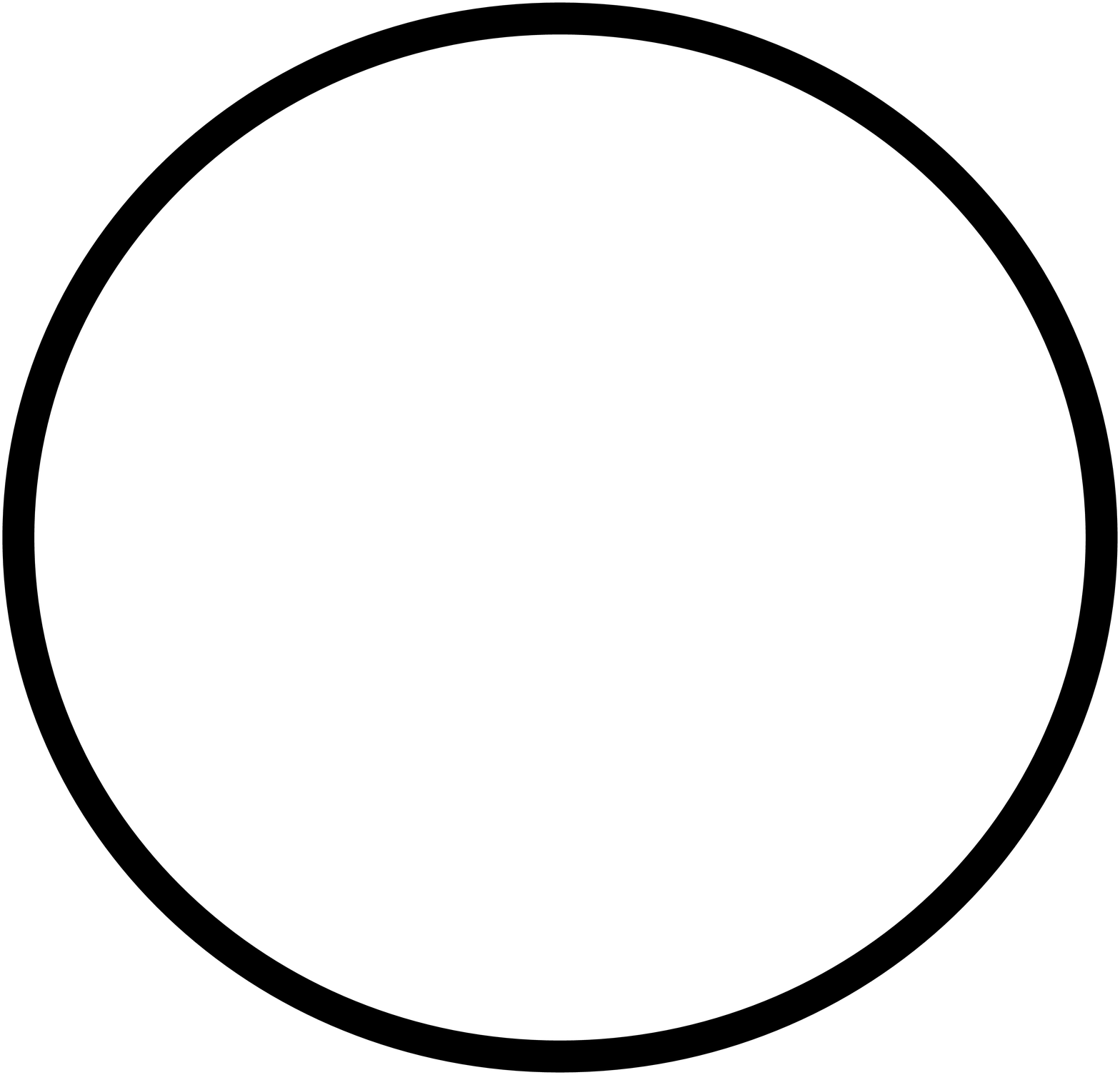}}
   \end{minipage}
  +
  (A^{2}+A^{-2})L,
  \end{eqnarray*} where $L\sqcup$ \begin{minipage}[h]{0.05\linewidth}
        \vspace{0pt}
        \scalebox{0.02}{\includegraphics{simple-circle}}
   \end{minipage} consists of an element $L$ in $\mathcal{L}_M$ and the trivial framed unknot, is called a \textit{weight relation}. Let denote by $R(M)$ the smallest submodule of $\tilde{\mathbb{Z}}[A,A^{-1}]$ that is generated by all possible skein relations and weight relations. The \textit{Kauffman bracket skein module} of $M$ is defined to be the quotient module
$\mathcal{S}(M)=\tilde{\mathbb{Z}}[A,A^{-1}]\mathcal{L}_M/R(M).$
The definition of the Kauffman bracket skein module can be extended to include $3$-manifolds with boundary. In this case we call the resulting module the \textit{relative Kauffman bracket skein module}. More precisely, the definition goes as follows. Specify a finite set of marked points on the boundary of $M.$ A band is
a surface that is homeomorphic to $I\times I$. An element in the set $\mathcal{L}_M$ is an isotopy class of an oriented surface embedded into $M$ and decomposed into a union of finite number of framed links and bands joining the
designated boundary points. The relative Kauffman bracket skein module is the quotient module $\mathcal{S}(M)=\tilde{\mathbb{Z}}[A,A^{-1}]\mathcal{L}_M/R(M).$ In this paper we are interested in
the case when $M$ is $F\times I$ where $F$ is an orientable surface. When this is the case, one thinks of a framed link in $M$ as a link diagram in $F$ with framings
being determined by parallel curves in $F$. The Kauffman bracket skein module of $F \times I$ will be denoted by $\mathcal 
{S}(F).$\\
In this paper we will use the Kauffman bracket skein module of $S^2$. This module is isomophic to $\tilde{\mathbb{Z}}[A,A^{-1}]$. To see this, let $D$ be a diagram in $S^2$. Using the definition of the normalized Kauffman bracket we can write $D=  <D>  \emptyset$ where $\emptyset$ is the empty link. This defines an isomorphism between $\mathcal{S}(S^2)$ and  $\tilde{\mathbb{Z}}[A,A^{-1}]$ induced by sending $D$ to $<D>$. The second module that we will use is the relative skein module $\mathcal{S}(I\times I,2 n)$ of the disk $I\times I$ with $n$ marked points on the top and $n$ points on the bottom. Two diagrams in $\mathcal{S}(I \times I,2n)$ $\ $can be concatenated to produce another diagram in $\mathcal{S}(I \times I, 2n).$ This defines a multiplication on $\mathcal{S}(I \times I,2n)$ that makes this module an associative unital algebra over $\tilde{\mathbb{Z}}[A,A^{-1}]$. This algebra is called the \textit{Temperley-Lieb} algebra and is denoted usually by $TL_n$.

For each $n \geq 1$ there exists a unique idempotent $f^{(n)}\in TL_n$ called the Jones-Wenzl idempotent. We will use a graphical notation for $f^{(n)}$ which is due to Lickorish \cite{Lic92}. In this graphical notation one thinks of $f^{(n)}$ as an empty box with $n$ strands coming in the top and $n$ strands leaving the bottom. The Jones-Wenzl idempotent enjoys a recursive formula that is due to Wenzl \cite{Wenzl}.  The recursive formula can be stated graphically as follows :
\begin{align}
\label{recursive}
  \begin{minipage}[h]{0.05\linewidth}
        \vspace{0pt}
        \scalebox{0.12}{\includegraphics{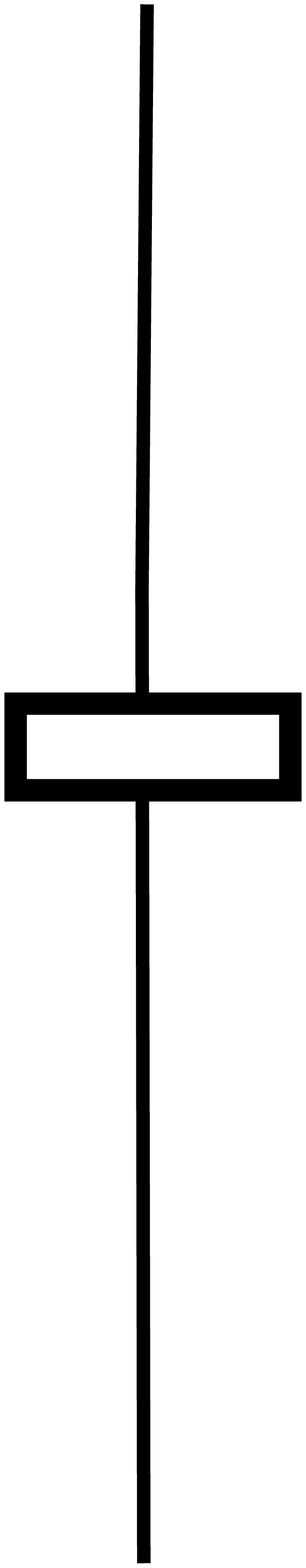}}
         \put(-20,+70){\footnotesize{$n$}}
   \end{minipage}
   =
  \begin{minipage}[h]{0.08\linewidth}
        \hspace{8pt}
        \scalebox{0.12}{\includegraphics{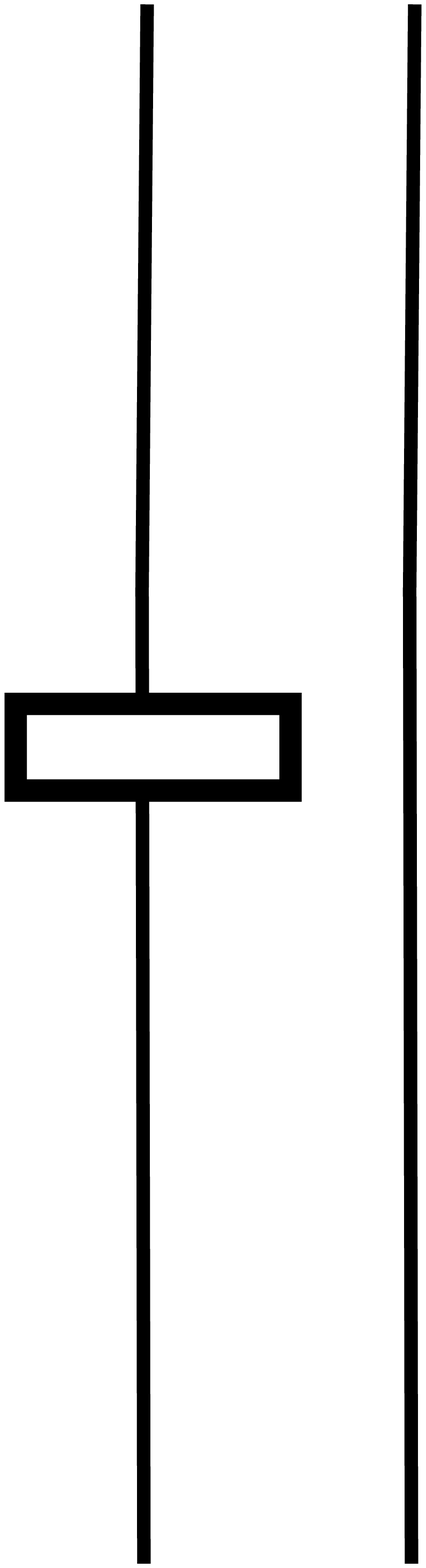}}
        \put(-42,+70){\footnotesize{$n-1$}}
        \put(-8,+70){\footnotesize{$1$}}
   \end{minipage}
   \hspace{9pt}
   -
 \Big( \frac{\Delta_{n-2}}{\Delta_{n-1}}\Big)
  \hspace{9pt}
  \begin{minipage}[h]{0.10\linewidth}
        \vspace{0pt}
        \scalebox{0.12}{\includegraphics{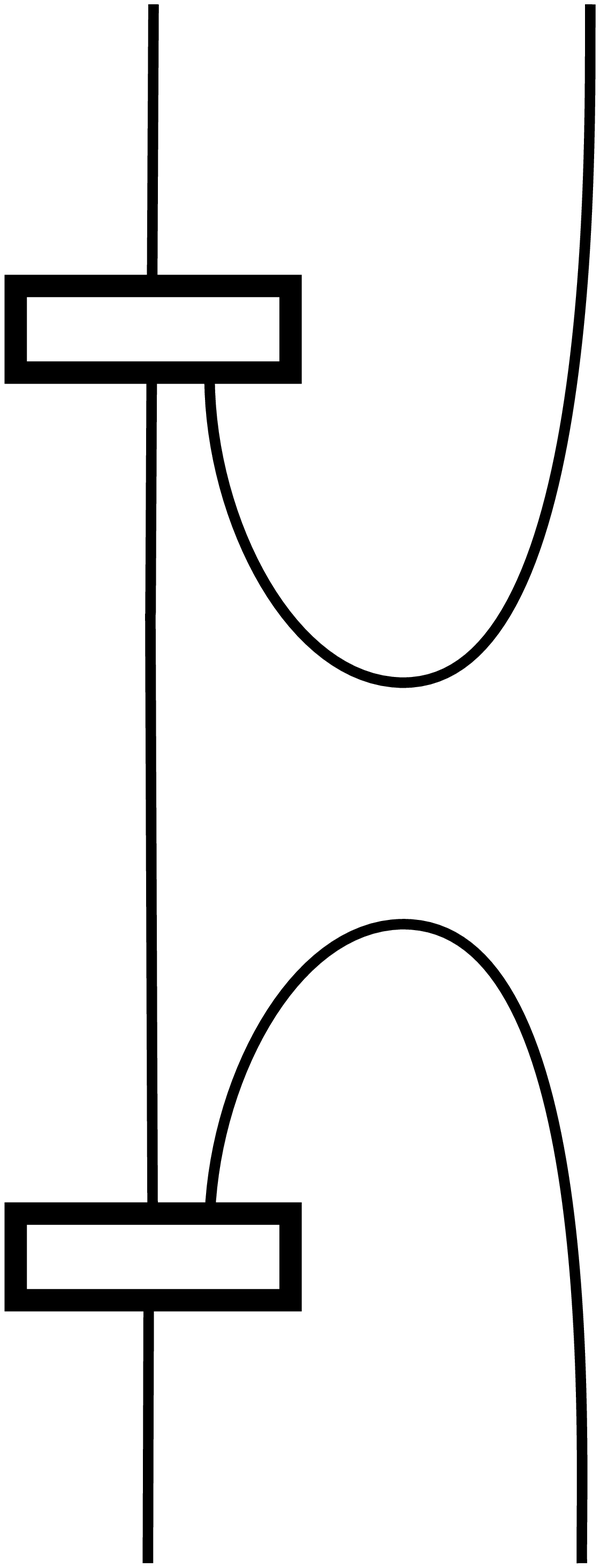}}
         \put(2,+85){\footnotesize{$1$}}
         \put(-52,+87){\footnotesize{$n-1$}}
         \put(-25,+47){\footnotesize{$n-2$}}
         \put(2,+10){\footnotesize{$1$}}
         \put(-52,+5){\footnotesize{$n-1$}}
   \end{minipage}
  , \hspace{20 mm}
    \begin{minipage}[h]{0.05\linewidth}
        \vspace{0pt}
        \scalebox{0.12}{\includegraphics{nth-jones-wenzl-projector}}
        \put(-20,+70){\footnotesize{$1$}}
   \end{minipage}
  =
  \begin{minipage}[h]{0.05\linewidth}
        \vspace{0pt}
        \scalebox{0.12}{\includegraphics{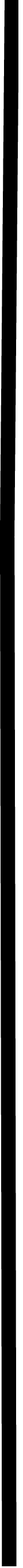}}
   \end{minipage}   
  \end{align}
  where 
\begin{equation*}
 \Delta_{n}=(-1)^{n}\frac{A^{2(n+1)}-A^{-2(n+1)}}{A^{2}-A^{-2}}.
\end{equation*} 

Furthermore, the idempotent $f^{(n)}$ has the following properties: 

\begin{eqnarray}
\label{properties}
\hspace{0 mm}
    \begin{minipage}[h]{0.21\linewidth}
        \vspace{0pt}
        \scalebox{0.115}{\includegraphics{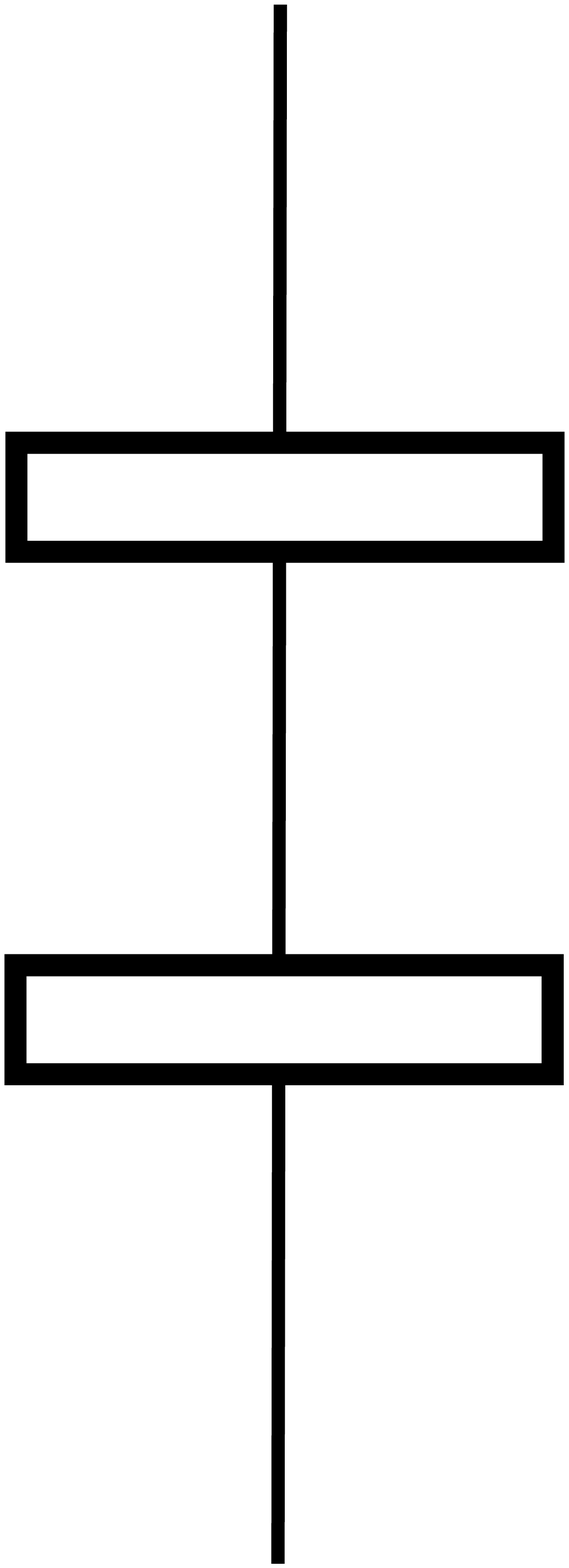}}
        \put(0,+80){\footnotesize{$n$}}
       
   \end{minipage}
  = \hspace{5pt}
     \begin{minipage}[h]{0.1\linewidth}
        \vspace{0pt}
         \hspace{50pt}
        \scalebox{0.115}{\includegraphics{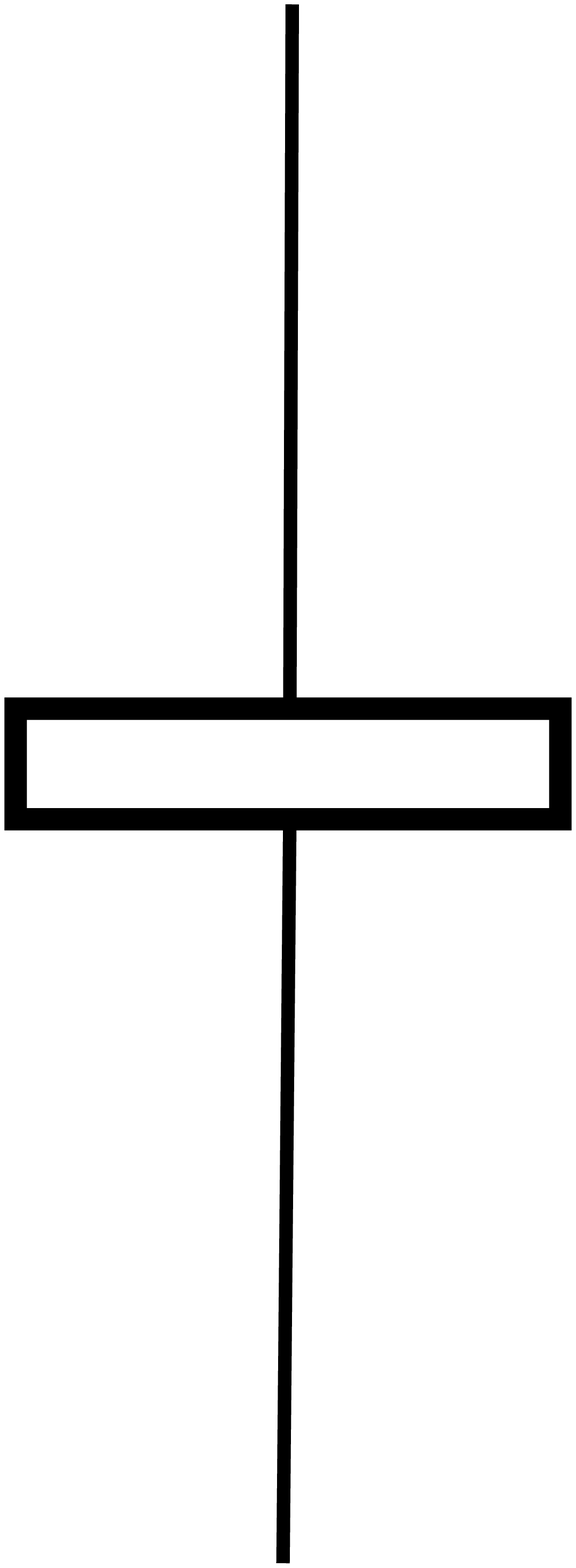}}
        \put(-60,80){\footnotesize{$n$}}
   \end{minipage}
    , \hspace{15 mm}
    \begin{minipage}[h]{0.09\linewidth}
        \vspace{0pt}
        \scalebox{0.115}{\includegraphics{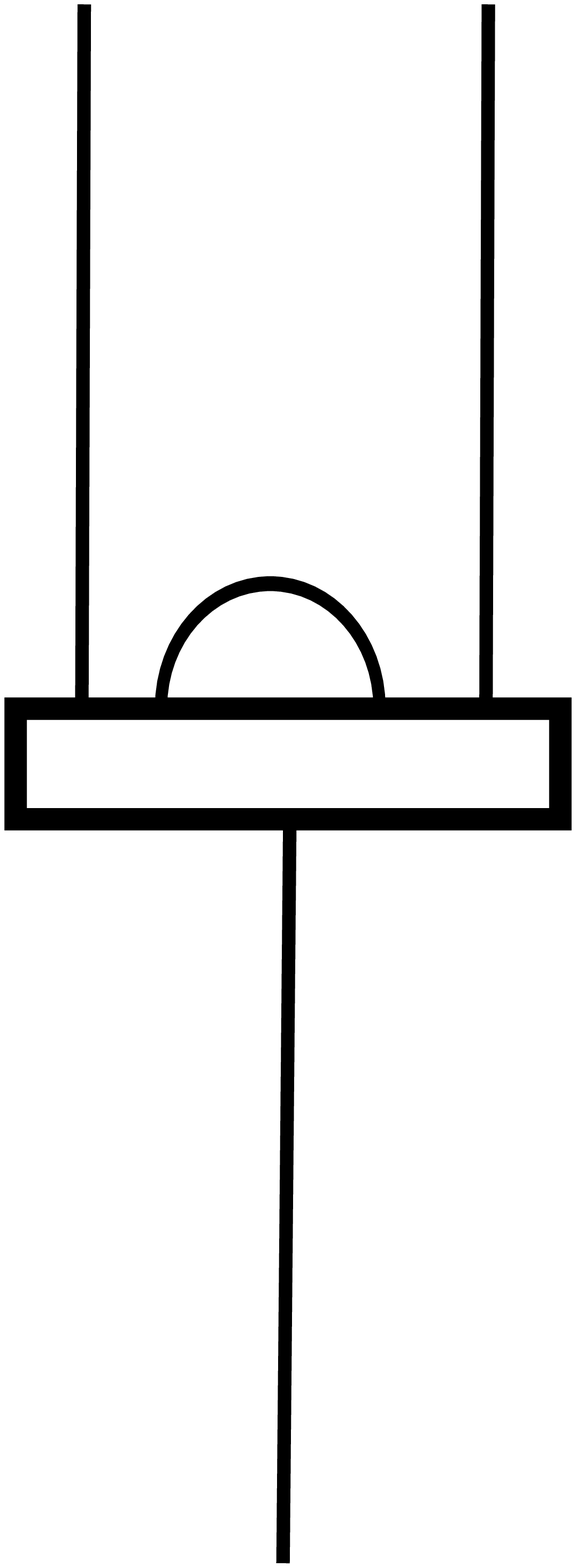}}
         \put(-70,+82){\footnotesize{$n-i-2$}}
         \put(-20,+64){\footnotesize{$1$}}
        \put(-2,+82){\footnotesize{$i$}}
        \put(-28,20){\footnotesize{$n$}}
   \end{minipage}
   =0,
   \label{AX}
  \end{eqnarray}
And:
\begin{eqnarray}
\label{properties}
\hspace{0 mm}
\Delta_{n}=
  \begin{minipage}[h]{0.1\linewidth}
        \vspace{0pt}
        \scalebox{0.12}{\includegraphics{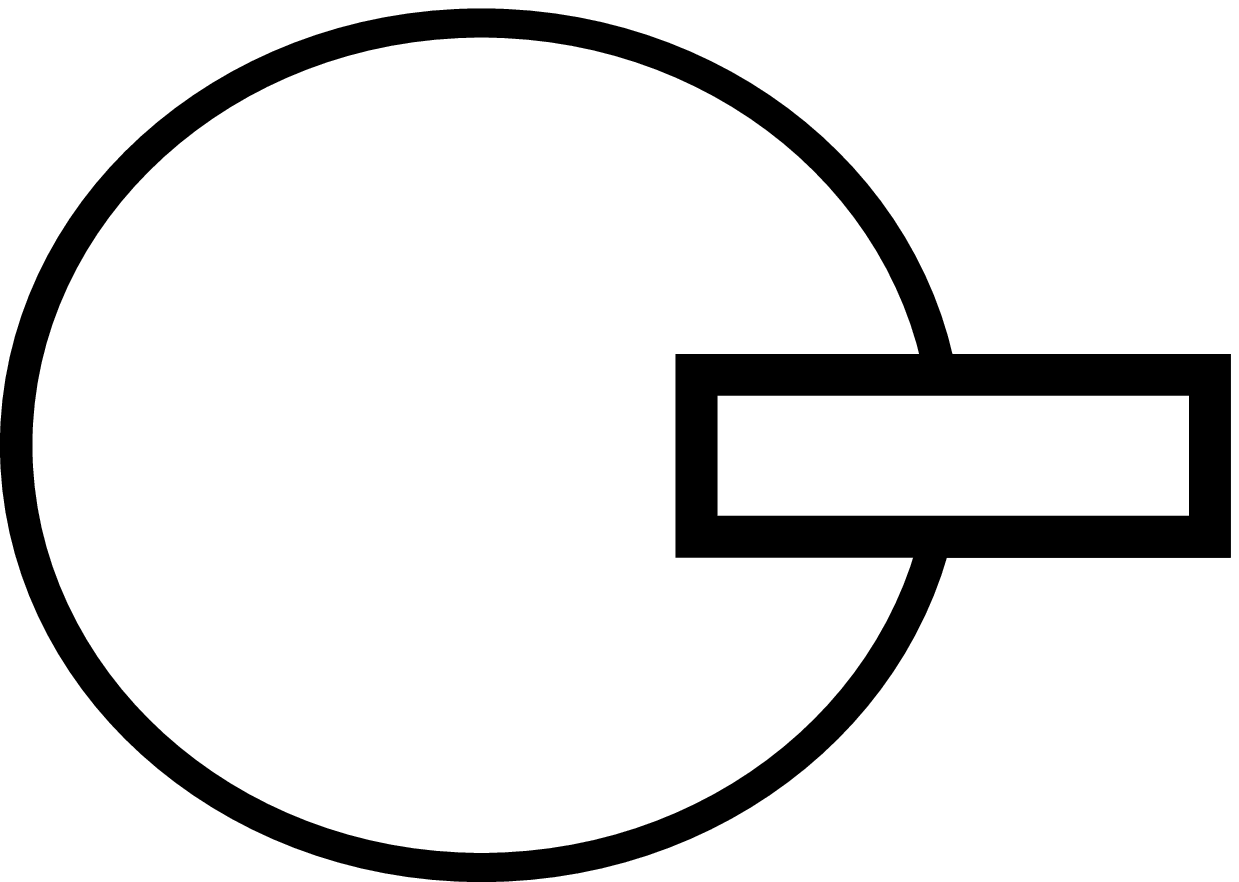}}
        \put(-29,+34){\footnotesize{$n$}}
   \end{minipage}
   , \hspace{14 mm}
     \begin{minipage}[h]{0.08\linewidth}
        \vspace{0pt}
        \scalebox{0.115}{\includegraphics{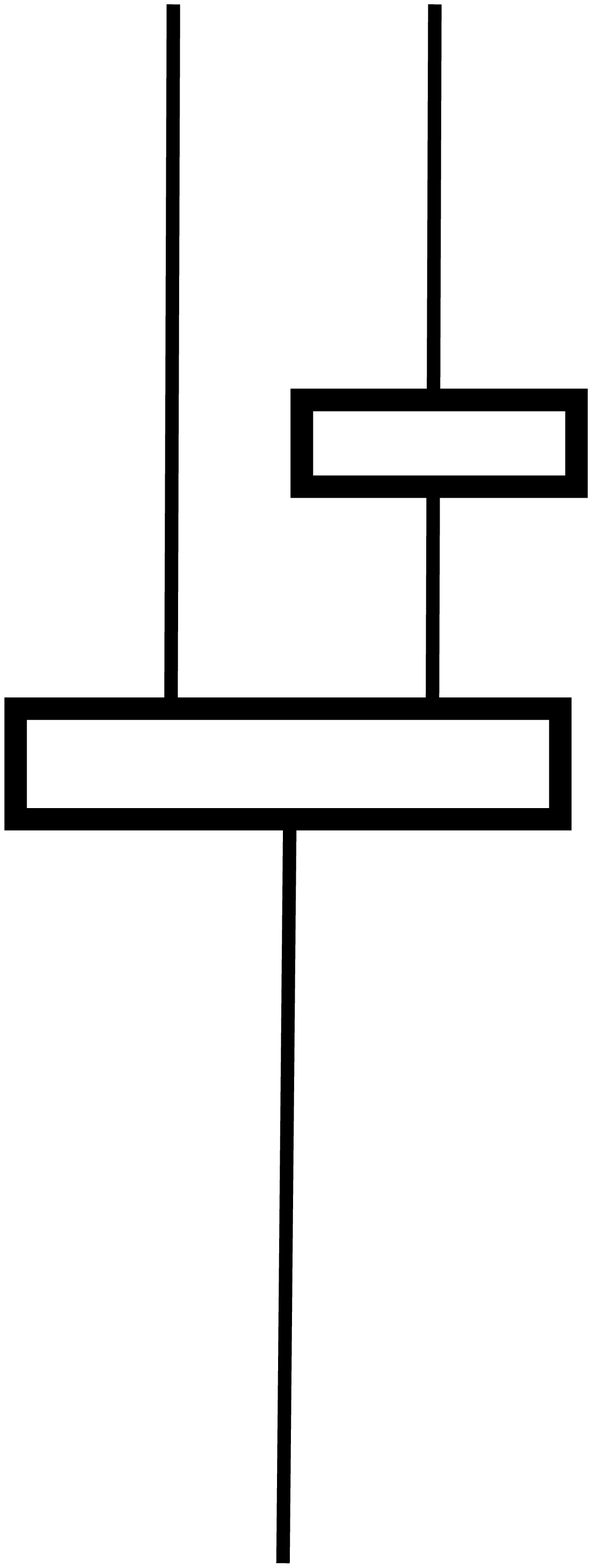}}
        \put(-34,+82){\footnotesize{$n$}}
        \put(-19,+82){\footnotesize{$m$}}
        \put(-46,20){\footnotesize{$m+n$}}
   \end{minipage}
  =
     \begin{minipage}[h]{0.09\linewidth}
        \vspace{0pt}
        \scalebox{0.115}{\includegraphics{idempotent2}}
        \put(-46,20){\footnotesize{$m+n$}}
   \end{minipage}
  \end{eqnarray}

If $A_n$ is a diagram in $TL_n$ and $B_n$ is a diagram in $TL_m$ then we define $A_n \otimes B_n $ to be the diagram in $TL_{(m+n)}$ obtained by joining the diagrams $A_n$ and $B_m$ as follows: 
\begin{eqnarray*}
   \begin{minipage}[h]{.5\linewidth}
         \vspace{0pt}
                  \scalebox{0.21}{\includegraphics{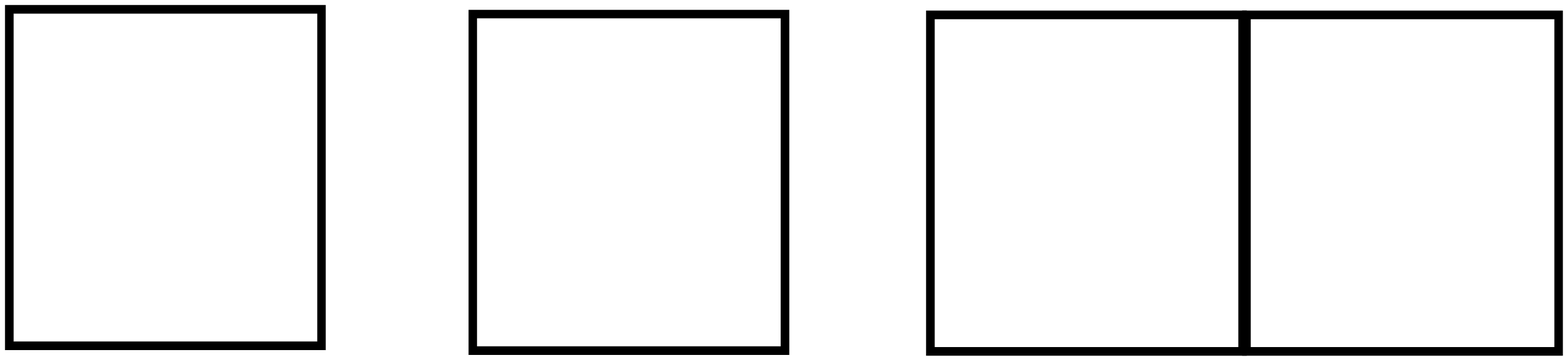}}         
    \put(-65,18){$A_n$}
         \put(-115,18){$B_m$}
          \put(-157,18){$A_n$}
          \put(-35,18){$B_m$}
          \put(-134,18){$\otimes$}
          \put(-88,18){$=$}
         \end{minipage}
  \end{eqnarray*}  
  
\noindent  
Before we introduce other skein modules we need the concept of \textit{wiring maps} to deal with linear maps between skein modules.
\subsubsection{Wiring Maps}  
We can relate various skein modules by linear maps induced from maps between surfaces. Let $F$ and $F^{\prime}$ be two oriented surfaces with marked points on their boundaries. A \textit{wiring} is an orientation preserving embedding of $F$ into $F^{\prime}$ along with a fixed \textit{wiring diagram} of arcs and curves in $F^{\prime}- F$ such that the boundary points of the arcs consists of all the marked points of $ F$ and $F^{\prime} $. Any diagram $D$ in $F$ induces a diagram $\mathcal{W}(D)$ in $F^{\prime}$ by extending $D$ by a wiring diagram. A wiring $W$ of $F$ into $F ^{\prime}$ induces a module homomorphism 
\begin{equation*}
\mathcal{S}(W):\mathcal{S}(F)\longleftrightarrow \mathcal{S}(F ^{\prime})
\end{equation*}
defined by $D \mapsto \mathcal 
{W}(D)$ for any $D$ diagram in $F$. More details about skein wiring can be found in the paper of Morton \cite{Morton}.  The following is an example which induces a map between the Temperley-Lieb algebra and the Kauffman skein module of the $2$-sphere.

\begin{example}
Consider the square $I \times I$ with $n$ marked points on the
top edge and $n$ marked points on the bottom edge. Embed $I \times I$ in $S^2$ and join the $n$ points on the top edge to the $n$ points on the bottom edge by parallel arcs as follows: 
\begin{eqnarray*}
   \begin{minipage}[h]{.5\linewidth}
         \vspace{0pt}
                  \scalebox{0.25}{\includegraphics{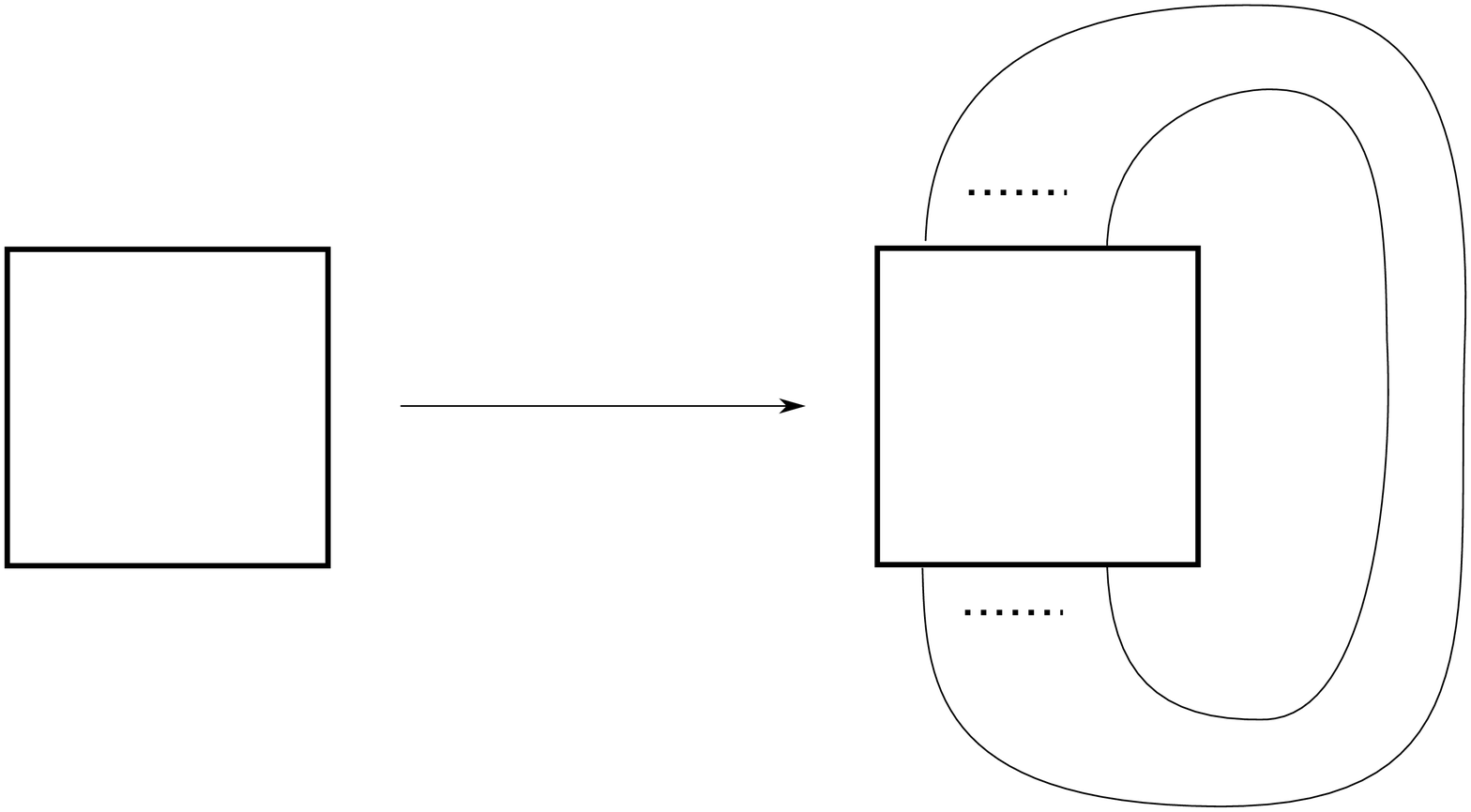}}
         \end{minipage}
  \end{eqnarray*}
 For each $n$, this wiring induces a module homomorphism:
  \begin{equation*}
	tr_n:TL_n \longrightarrow \mathcal{S}(S^2)  
  \end{equation*}
  This map is usually called the \textit{Markov trace} on $TL_n$.
\end{example}
\subsection{Colored Trivalent Graphs}
Now we will define certain submodules of the skein module of the disk $I \times I$ with marked points on the boundary. These modules will be useful in the computation of the tail of the colored Jones polynomial. Consider the skein module of $I \times I$ with $a+b+c$ specified points on the boundary. Partition the set
of the $a+b+c$ points on the boundary of the disk into $3$ sets of  $a$, $b$ and  $c$  
points respectively and at each cluster of points we place an appropriate idempotent, i.e. the one whose color matches the cardinality of this cluster. The skein module constructed in this method will be denoted by $T_{a,b,c}.$ The skein module $T_{a,b,c}$ is either zero dimensional or one dimensional. The skein module $T_{a,b,c}$ is one dimensional if and only if the element shown in Figure \ref{taw} exists. For this element to exist it is necessary to find
non-negative integers $x,y$ and $z$ such that $a=x+y$, $b=x+z$ and $c=y+z$. 
\begin{figure}[H]
  \centering
   {\includegraphics[scale=0.2]{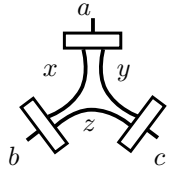}
   \small{
    \put(-60,-4){$b$}
         \put(-34,53){$a$}
          \put(-47,30){$x$}
          \put(-19,30){$y$}
          \put(-32,9){$z$}
         \put(-5,-4){$c$}
         }
     \caption{The skein element $\tau_{a,b,c}$ in the space $T_{a,b,c\text{ }}$ }
  \label{taw}}
\end{figure}
The following definition characterizes the existence of this skein element in terms of the integers $a$, $b$ and $c$. 
\begin{definition}
\label{admi}
A triple of non-negative integers $(a,b,c)$ is \textit{admissible} if $a+b+c$ is even and $a+b\geq
c\geq |a-b|.$
\end{definition}

When the triple $(a,b,c)$ is admissible, one can write  
$x=\frac{a+b-c}{2}$,
$y=\frac{a+c-b}{2}$, and 
$z=\frac{b+c-a}{2}$.  In this case we will denote the skein element that generates the space by $\tau_{a,b,c}$. Note that when the triple $(a,b,c)$ is not admissible then the space $T_{a,b,c}$ is zero dimensional.  The fact that the inside colors are determined by the outside colors allows us to replace $\tau _{a,b,c}$ by a trivalent graph as follows:
\begin{figure}[H]

  \centering
   {\includegraphics[scale=0.25]{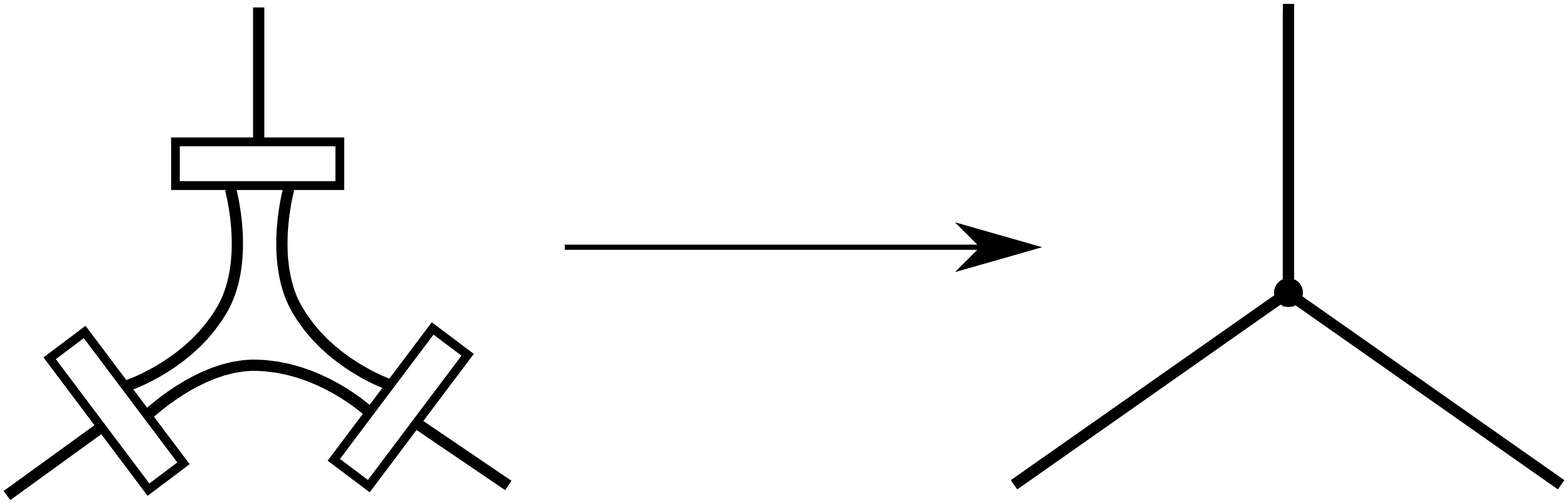}
    \put(-90,-8){$b$}
         \put(-38,68){$a$}
          \put(-227,30){$x$}
          \put(-199,30){$y$}
          \put(-212,10){$z$}
         \put(-5,-8){$c$}
            \put(-260,-8){$b$}
         \put(-205,68){$a$}
               \put(-169,-8){$c$}
     \caption{}
  \label{taw1}}
\end{figure}

Similarly, we define the module of the disk $\mathscr{D}^{a,b}_{c,d}$. Precisely the skein module $\mathscr{D}^{a,b}_{c,d}$ is a submodule of the skein module of the disk with $a+b+c+d$ marked points on the boundary and we place the idempotents $f^{(a)}$, $f^{(b)}$, $f^{(c)}$, and $f^{(d)}$ on the appropriate set of points as we did for $T_{a,b,c}$. See Figure \ref{disk}.
   \begin{figure}[H]
  \centering
   {\includegraphics[scale=0.13]{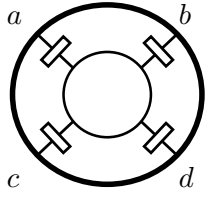}
    \put(-75,+62){$a$}
          \put(-10,+62){$b$}
          \put(-75,0){$c$}
          \put(-10,0){$d$}
            \caption{The relative skein module $\mathscr{D}^{a,b}_{c,d}$}
  \label{disk}}
\end{figure}

In order to perform our computation for the tail of the colored Jones polynomial it is important to understand the evaluation of certain skein elements in $\mathcal{S}(S^2)$. Th evaluation of these skein elements can be understood as the evaluation of certain colored trivalent graphs in $\mathcal{S}(S^2)$. A \textit{colored trivalent graph} is a planer trivalent graph with edges labeled by non-negative integers. One usually uses the word \textit{color} to refer to a label of the edge of a trivalent graph. A colored trivalent graph is called \textit{admissible} if the three edges meeting at a vertex satisfy the admissibility condition of the definition \ref{admi}. If $D$ is an admissible colored trivalent graph then the Kauffman bracket evaluation of $D$ is defined to be the evaluation of $D$ as an element in $\mathcal{S}(S^{2})$ after replacing each edge colored $n$ by the projector $f^{(n)}$ and each admissible vertex colored $(a,b,c)$ by the skein element $\tau_{a,b,c}$, as in Figure \ref{taw1}. If a colored trivalent graph has an inadmissible vertex then we will consider its evaluation in $\mathcal{S}(S^{2})$ to be zero.  We will need the evaluation of the following important colored trivalent graphs shown in Figure \ref{graphs}.

   \begin{figure}[H]
  \centering
   {\includegraphics[scale=0.13]{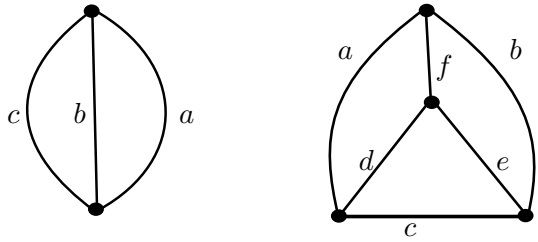}
    \put(-75,+62){$a$}
          \put(-10,+62){$b$}
          \put(-50,-5){$c$}
          \put(-67,20){$d$}
          \put(-15,20){$e$}
          \put(-38,56){$f$}
          \put(-135,38){$a$}
          \put(-175,38){$b$}
          \put(-200,38){$c$}
            \caption{The theta graph on the left and the tetrahedron graph on the right.}
  \label{graphs}}
\end{figure}
For an admissible triple $(a,b,c)$, an explicit formula for the \textit{theta coefficient}, denoted $\Theta(a,b,c)$, was computed in \cite{MV} and is given by:
\begin{equation}
  \begin{minipage}[h]{0.1\linewidth}
       \scalebox{0.10}{\includegraphics{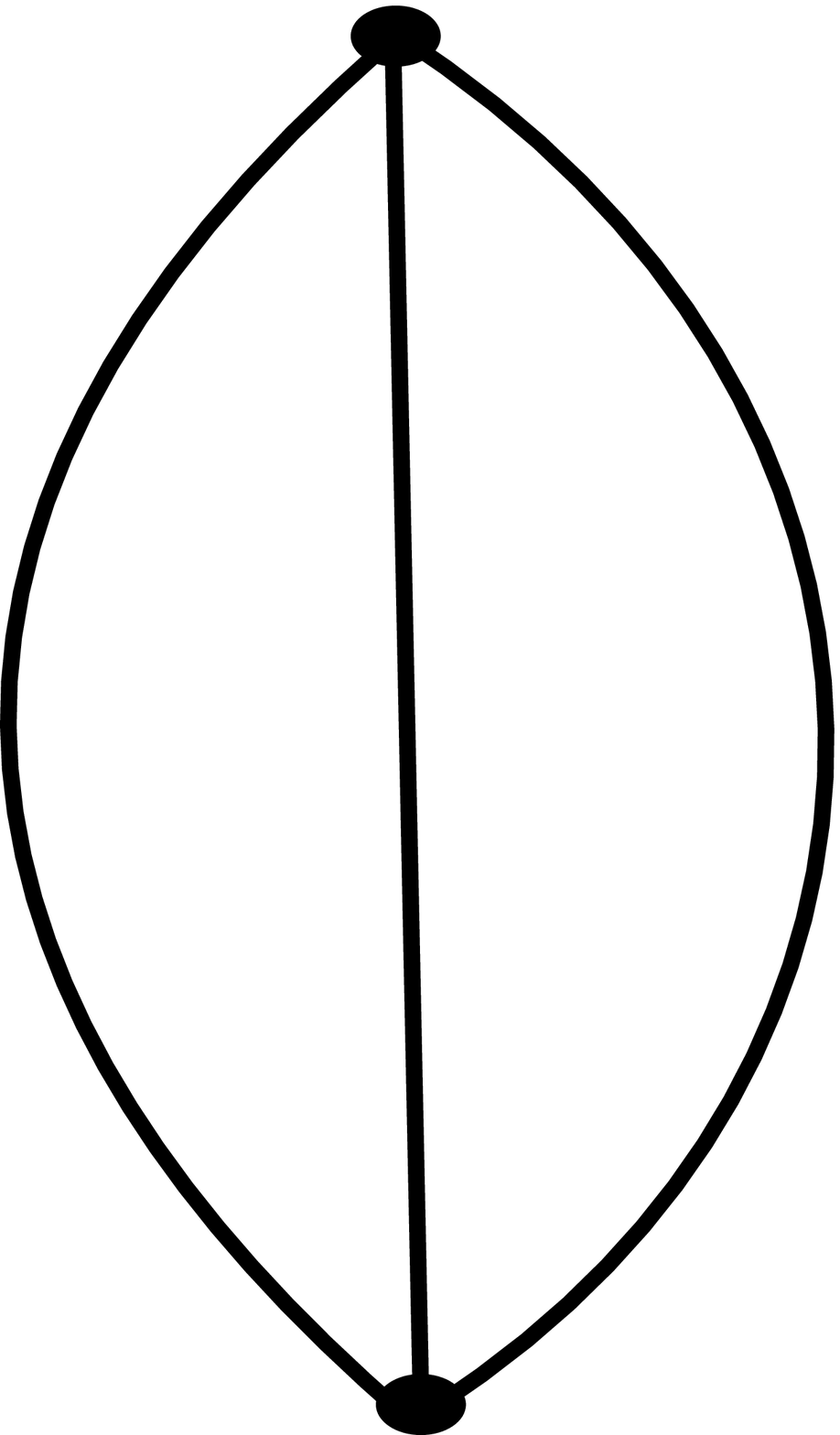}}
       \put(-42,25){$a$}
       \put(-24,25){$b$}
       \put(-10,25){$c$}
   \end{minipage}
   =(-1)^{i+j+k}\frac{[i+j+k+1]![i]![k]![j]!}{[i+j]![i+k]![j+k]!}
   \end{equation}
where $i,j$ and $k$ are the interior colors of the vertex $(a,b,c)$. The  \textit{tetrahedron coefficient} is defined to be the evaluation of the graph appearing on the right handside of Figure \ref{graphs} and a formula of it can be found in \cite{MV}. The tetrahedron graph in Figure \ref{graphs} is denoted by $Tet\left[ 
\begin{array}{ccc}
a & d & e \\ 
f & c & b%
\end{array}%
\right]$. The following identity holds in $T_{a,b,c}$:

\begin{eqnarray}
\label{firsty}
    \begin{minipage}[h]{0.21\linewidth}
         \vspace{-0pt}
         \scalebox{0.35}{\includegraphics{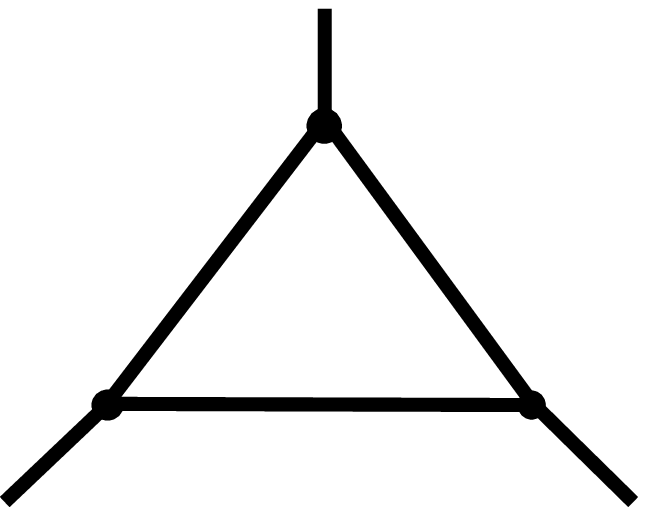}}
     \put(-69,-10){$b$}
          \put(-1,-10){$c$}
         \put(-26,50){$a$}
         \put(-59,25){$d$}
         \put(-16,25){$e$}
          \put(-36,0){$f$}
                    \end{minipage}&=&\frac{Tet\left[ 
\begin{array}{ccc}
a & d & e \\ 
f & c & b%
\end{array}%
\right]}{\Theta(a,b,c)}
   \begin{minipage}[h]{0.16\linewidth}
        \vspace{-0pt}
        \scalebox{0.25}{\includegraphics{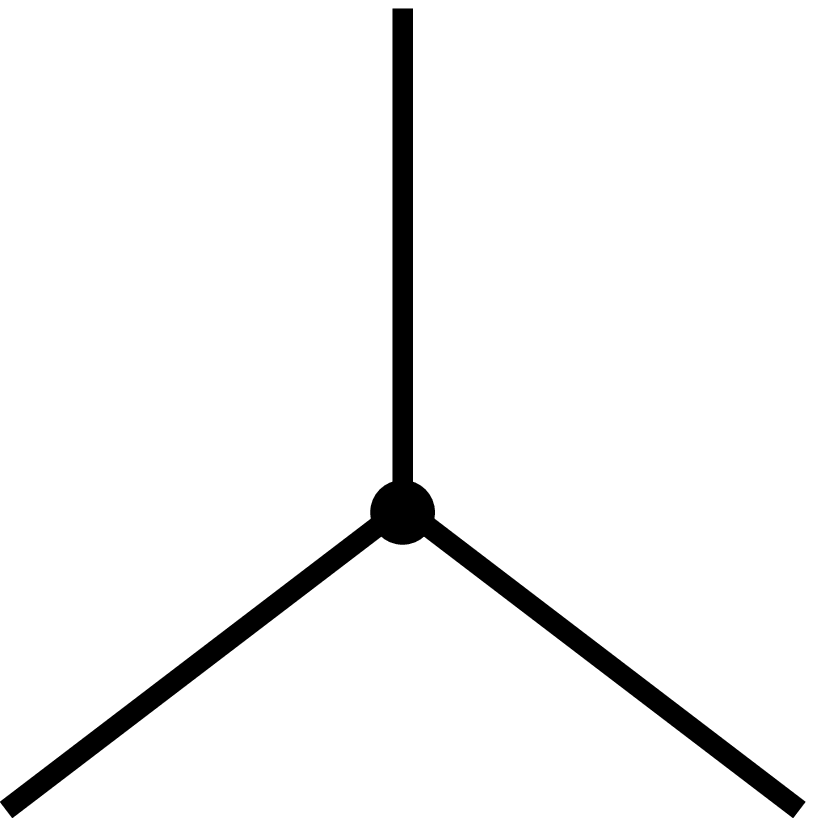}}
         \put(-60,-10){$b$}
          \put(-1,-10){$c$}
         \put(-26,50){$a$}
           \end{minipage}
  \end{eqnarray}

Recall that, for any integers $l,i$ such that $ 0\leq i \leq l$, the quantum binomial coefficients are defined by :
\begin{equation*}
{l \brack i}_{q}=\frac{(q;q)_l}{(q;q)_i(q;q)_{l-i}}.
\end{equation*}
where $(a;q)_n$ is $q$-Pochhammer symbol which is defined as 
\begin{equation*}
(a;q)_n=\prod\limits_{j=0}^{n-1}(1-aq^j).
\end{equation*}
 We will need the following identity \cite{Hajij1}.  
\begin{theorem}(The bubble expansion formula)
\label{main111}
Let $m,n,m^{\prime},n^{\prime} \geq 0$, and $k\geq l$;  $k,l \geq 1$. Then\\
{\small
\begin{eqnarray}
\label{bubble expansion formula121}
  \begin{minipage}[h]{0.17\linewidth}
        \vspace{0pt}
        \scalebox{0.11}{\includegraphics{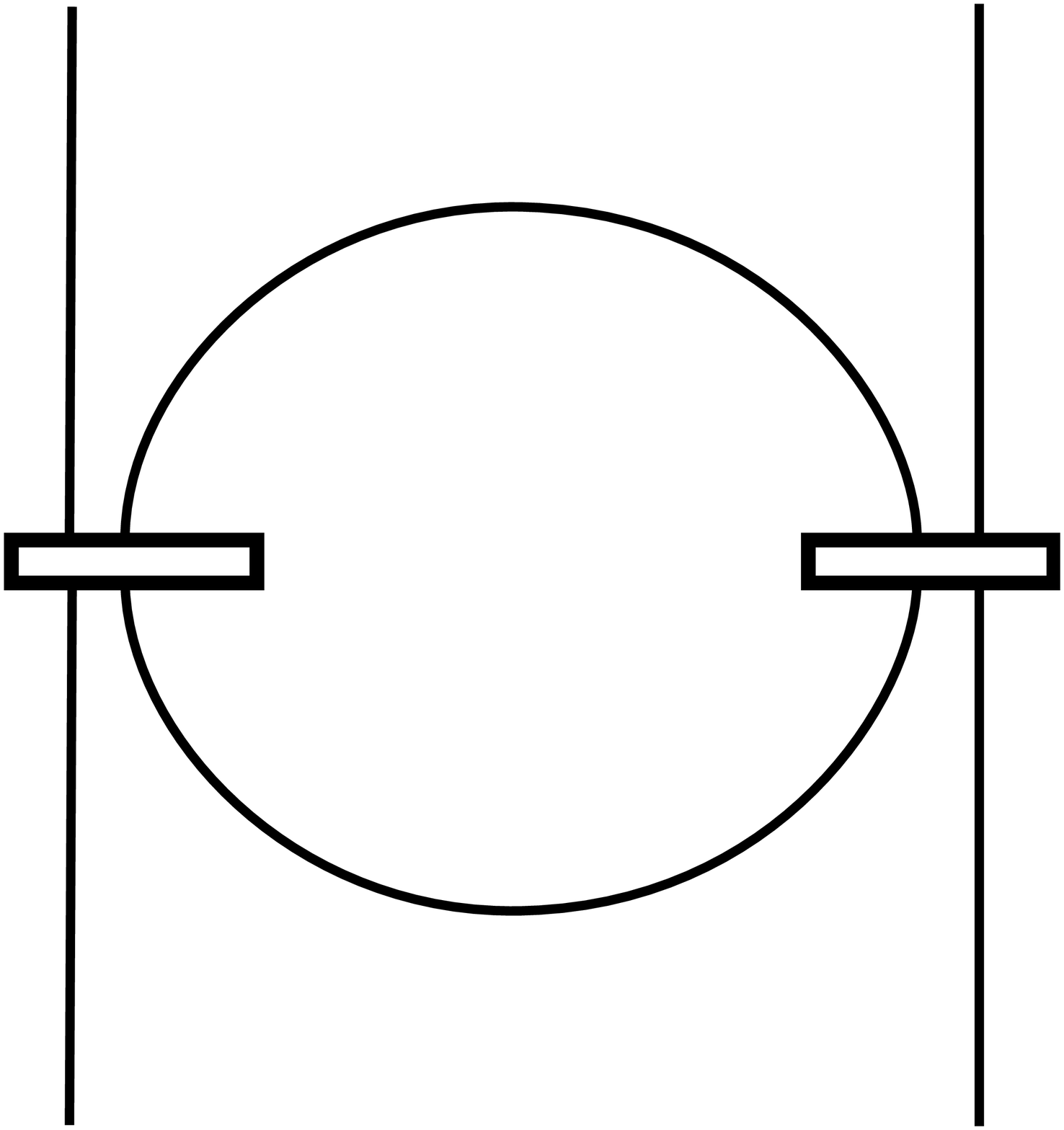}}
        \put(-68,+80){$m$}
          \put(-8,+80){$n$}
          \put(-68,-7){$m^{\prime}$}
          \put(-8,-7){$n^{\prime}$}
           \put(-37,+67){$k$}
         \put(-37,+3){$l$}
   \end{minipage}
   =\displaystyle\sum\limits_{i=0}^{\min(m,n,l)}
   \left\lceil 
\begin{array}{cc}
m & n \\ 
k & l%
\end{array}%
\right\rceil _{i}
    \begin{minipage}[h]{0.15\linewidth}
        \vspace{0pt}
        \scalebox{0.11}{\includegraphics{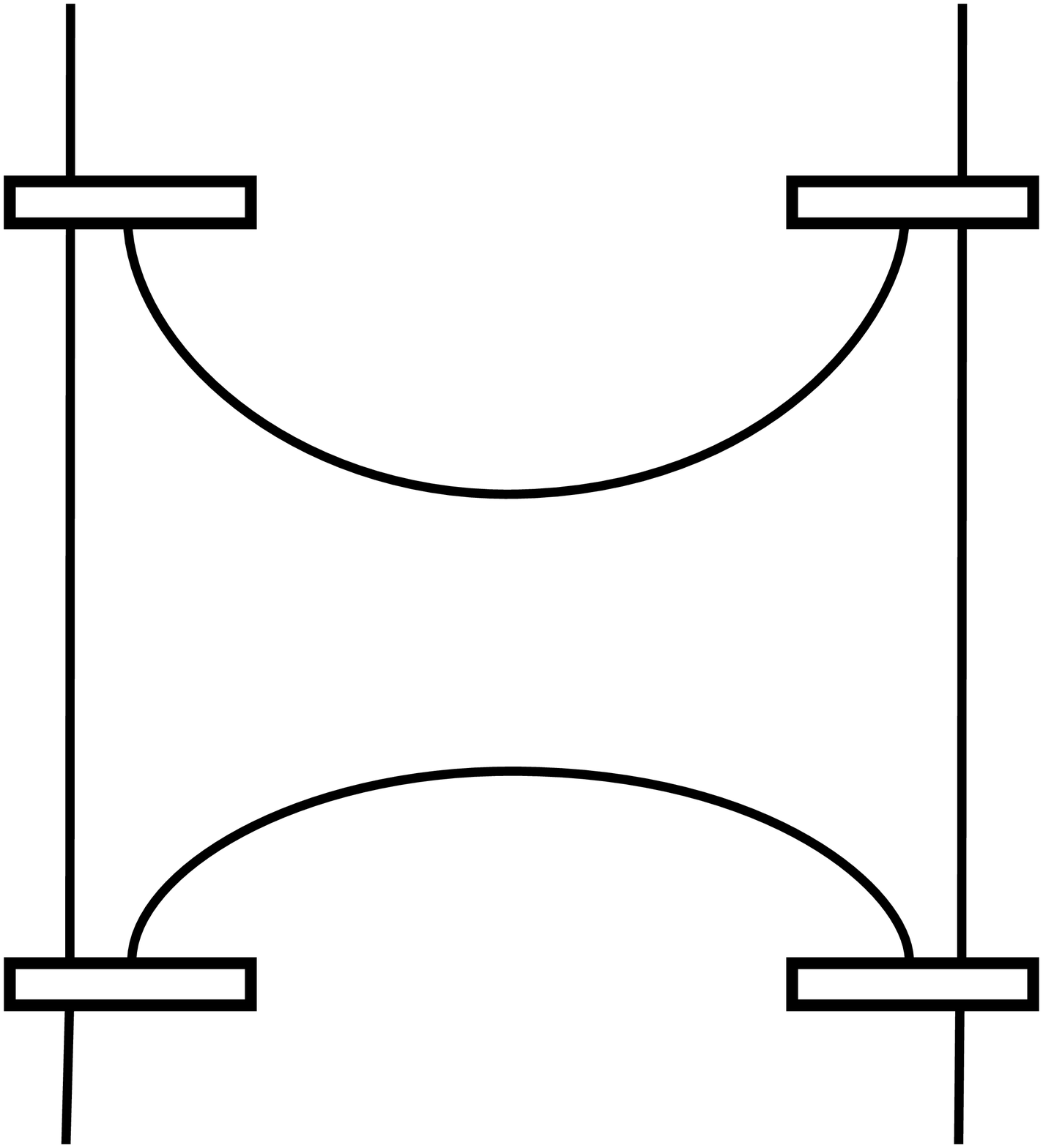}}
        \put(-68,+80){$m$}
          \put(-8,+80){$n$}
          \put(-68,-7){$m^{\prime}$}
          \put(-8,-7){$n^{\prime}$}
           \put(-37,+50){$i$}
         \put(-52,28){$k-l+i$}
   \end{minipage}.
  \end{eqnarray}
}
where 
\begin{equation*}
\left\lceil 
\begin{array}{cc}
m & n \\ 
k & l%
\end{array}%
\right\rceil_{i}:=(-A^2)^{i(i-l)}\frac{\displaystyle\prod_{j=0}^{l-i-1}\Delta
_{k-j-1}\prod_{s=0}^{i-1}\Delta _{n-s-1}\Delta _{m-s-1}}{%
\displaystyle\prod_{t=0}^{l-1}\Delta _{n+k-t-1}\Delta _{m+k-t-1}}{l \brack i}_{A^4}\prod_{j=0}^{l-i-1}\Delta_{m+n+k-i-j}.
\end{equation*}
\end{theorem}
We will denote the skein element on the right handside of (\ref{bubble expansion formula121}) by $\mathcal{B}^{m,n}_{m^{\prime},n^{\prime}}(k,l)$ and we will call it the \textit{bubble skein element}.
\subsection{The Tail of The Colored Jones Polynomial}
 We briefly review the basics of the head and the tail of the colored Jones polynomial. For more details see \cite{Hajij1, Hajij2}.
 
 Let $L$ be a framed link in $S^3$. Decorate every component of $L$, according to its framing, by the $n^{th}$ Jones-Wenzl idempotent and consider the evaluation of the decorated framed link as an element of $\mathcal{S}(S^3)$. Up to a power of $\pm A$, that depends on the framing of $L$, the value of this element is defined to be the $n^{th}$ (unreduced) colored Jones polynomial $\tilde{J}_{n,L}(A)$. Recovering the reduced Jones polynomial is a matter of changing a variable and dividing by $\Delta_n$. Namely,
 \begin{equation}
 \label{change of variable}
 J_{n+1,L}(q)=\frac{\tilde{J}_{n,L}(A)}{\Delta_n}\bigg|_{A=q^
 	{1/4}}
 \end{equation}
 
 \noindent
 If $P_1(q)$ and $P_2(q)$ are elements in $\mathbb{Z}[q^{-1}][[q]]$, we write $P_1(q)\doteq_n P_2(q)$ if their first $n$ coefficients agree up to a sign.  It was proven in \cite{CodyOliver} that the coefficients of the colored Jones polynomial of an alternating link $L$ stabilize in the following sense: For every $n\geq2$, we have $J_{n+1,L}(q)\doteq_n J_{n,L}(q)$. This motivated the authors of \cite{CodyOliver} to define the tail of the colored Jones polynomial of a link. More precisely, define the $q$-series series associated with the colored Jones polynomial of an alternating link $L$ whose $n^{th}$ coefficient is the $n^{th}$ coefficient of $J_{n,L}(q)$. Stated differently, the tail of the colored Jones polynomial of a link $L$ is defined to be a series
 $T_L(q)$, that satisfies  $T_L(q)\doteq_{n}J_{n,L}(q)$ for all $n \geq 1$. In the same way, the head of the colored Jones polynomial of a link $L$ is defined to be the tail of $J_{n,L}(q^{-1})$. 
 The head and the tail of the colored Jones polynomial of an alternating link $L$ can be recovered from a sequence of skein elements in $\mathcal{S}(S^2)$. The study of this sequence of skein elements is relatively easier than the study of the entire colored Jones polynomial. For more details see \cite{CodyOliver} and \cite{Hajij2}. We recall this fact here. Let $L$ be a link in $S^3$ and $D$ be an alternating knot diagram of $L$. Consider the all $B$-smoothings state of $D$, the state obtained by replacing each crossing by a $B$-smoothing. We record the places of this smoothing by a dashed line as can be seen in Figure \ref{allB} for an example. Write $S^{(n)}_B(D)$ for the all $B$-smoothing state and consider the skein element obtained from $S_B(D)$ by decorating each circle in $S_B(D)$ with the $n^{th}$ Jones-Wenzl idempotent and replacing each dashed line in $S_B(D)$ with the $(2n)^{th}$ Jones-Wenzl idempotent. See Figure \ref{allB}. 
 
 \begin{figure*}[htb]
 	\centering
 	{\includegraphics[scale=0.1]{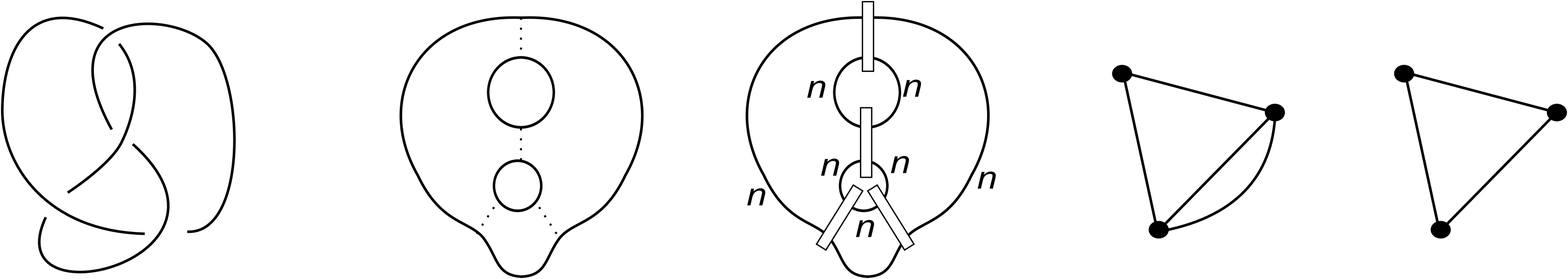}
 		\put(-380,-15){$D$ }
 		\put(-10,-15){$\mathbb{G}^{\prime}_B$  }
 		\put(-280,-15){$S_B(D)$ }
 		\put(-200,-15){$S^{(n)}_B(D)$ }
 		\put(-100,-15){$\mathbb{G}_B$ }
 		\caption{A link diagram $D$, its all-$B$ state $S_B(D)$, the skein element $S^{(n)}_B(D)$, the $B$-graph $\mathbb{G}_B(D)$, and the reduced all $B$-graph $\mathbb{G}^{\prime}_B(D)$.  }
 		\label{allB}}
 \end{figure*}
 The following theorem from \cite{CodyOliver} relates the tail of the colored Jones polynomial of an alternating link $D$ to the skein element $S_B^{(n)}(D)$.
 \begin{theorem}
 	\label{cody thm}
 	Let $L$ be an alternating link in $S^3$ and let $D$ be an alternating diagram of $L$. Then
 	\begin{equation*}
 	\tilde{J}_{n,L}(A)\doteq_{4(n+1)}S_B^{(n)}(D)
 	\end{equation*}  
 \end{theorem}
 This theorem states basically that the study of the tail of the colored Jones polynomial of the alternating knot $D$ can be reduced to the study of the tail of the  sequence of skein elements $S_B^{(n)}(D)$. This theorem also implies that the tail of the colored Jones polynomial depends on the so called the \textit{reduced B-graph} of the diagram $D$. The $B$-graph of the diagram $D$, denoted $\mathbb{G}_{B}(D)$ is the graph whose vertices are the circles of $S_B(D)$ and whose edges are the dashed lines. The reduced $B$-graph of $D$, denoted by $\mathbb{G}^{\prime}_{B}(D)$, is obtained from $\mathbb{G}_{B}(D)$ by replacing parallel edges by a single
 edge. See the most right two drawings in Figure \ref{allB}.
 \begin{remark}
 	Since the colored Jones polynomial of a diagram $D$ depends only on its reduced $B$-graph, we will sometimes use the term \textit{the tail a graph $G$} to refer to the tail of colored Jones polynomial of an alternating knot diagram $D$ such that $\mathbb{G}^{\prime}_B(D)=G$. Conversely, Given a planar graph $G$, we can obtain an alternating knot diagram $D$ such that $\mathbb{G}^{\prime}_B(D)=G$ by replacing every edge in $G$ by a crossing as illustrated in Figure \ref{graph2knot}. For this reason, if $G$ is a planar graph then the tail of $G$ will be denoted by $T_G$. Furthermore, the notation $S_B^{(n)}(G)$ will refer to the skein element obtained from the reduced graph $G$ by replacing each vertex with a circle and each edge with the $2n^{th}$ Jones-Wenzl projector. 
 \end{remark}

 \begin{figure}[H]
 	\centering
 	{\includegraphics[scale=0.25]{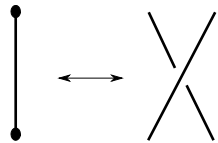}
 		\caption{Obtaining an alternating knot from a graph.}
 		\label{graph2knot}}
 \end{figure}
 \begin{remark}
 	\label{change remark}
 	In general the computations of the tail of the colored Jones polynomial is done for the reduced case. In order to use Theorem \ref{cody thm} one needs to do a change of variable and normalize by $\Delta$ as can be seen from the relation (\ref{change of variable}).
 \end{remark}
 The tail of the colored Jones polynomial has been computed for all knots in the knot table up to the knot $8_4$ by Armond and Dasbach in \cite{CodyOliver}. In \cite{Hajij2}, the second author gave a formula for $8_5$.
 
\section{Main Results}
\label{sec2.5}

In this section we list the main results of the paper. Let $a_1,\ldots,a_n$ be positive integers. Denote by $P(a_1,\ldots,a_n)$ the pretzel knot with $n$ crossing regions given in Figure \ref{pretzel family}.

\begin{figure}[H]
  \centering
   {\includegraphics[scale=0.25]{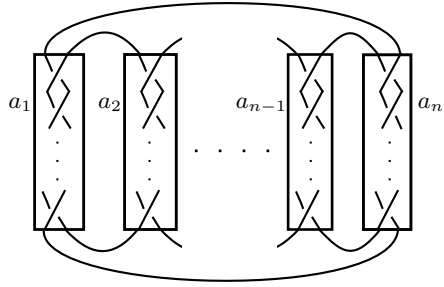}}
   \put(-153,+66){\footnotesize{$a_1$}}
   \put(-119,+66){\footnotesize{$a_2$}}
    \put(-67,+66){\footnotesize{$a_{n-1}$}}
    \put(2,+66){\footnotesize{$a_{n}$}}
   \caption{Pretzel knot $P(a_1,\ldots,a_n)$}
   \label{pretzel family}
\end{figure}

In the following theorem, we give a formula for the tail of the colored Jones polynomial of the pretzel knot $P(2k+1,2,2u+1)$ for $u,k \geq 1$.

\begin{theorem}
The tail of the pretzel knot $P(2k+1,2,2u+1)$ is given by
\begin{equation*}
T_{P(2k+1,2,2u+1)}(q)=(q;q)^2_\infty\sum\limits_{l_1=0}^{\infty}...\sum\limits_{l_k=0}^{\infty}\sum\limits_{p_1=0}^{\infty}...\sum\limits_{p_u=0}^{\infty}g(q;l_1,...,l_k)g(q;p_1,...,p_u)(q;q)_{i+j}
\end{equation*}
where 
\begin{equation*}
g(q;l_1,...,l_k)=\frac{q^{\sum\limits_{j=1}^{k}(i_j(i_j+1))}}{(q;q)^2_{l_{k}}\prod\limits_{j=1}^{k-1}(q;q)_{l_j}}
\end{equation*}
with $i_j=\sum\limits_{s=j}^{k}l_s$.
\end{theorem}
This formula generalizes the one of the tail of colored Jones polynomial of the knot $8_5$ given in \cite{Hajij1}. Furthermore, we give a formula for the tail of the colored Jones polynomial of the pretzel knot $P(2,\ldots,2)$ with $k+1$ crossing regions. 

\begin{proposition}
\label{side1}
Let $k \geq 1 $ and let $P_k$ denotes $P(2,\ldots,2)$ with $k+1$ crossing regions. Then 
\begin{equation*}
T_{P_k}(q)=(q;q)_{\infty}^k \sum_{i=0}^{\infty}\frac{q^{i}}{(q;q)_{i}^k}.
\end{equation*}
\end{proposition}
We use skein theoretic techniques  
to give another method to compute $T_{P_k}(q)$ and we obtain the following identity.
\begin{corollary}
For $k \geq 1 $ we have
\begin{equation*}
(q;q)_{\infty} \sum_{i=0}^{\infty}  \frac{q^{i}}{(q;q)_i^{k+1}} = \sum_{i_1=0}^{\infty}...\sum_{i_k=0}^{\infty}\frac{q^{\sum_{j=1}^k i_j+i_j^2+\sum_{s=2}^k \sum_{j=s}^k i_{s-1}i_j}}{\prod_{j=1}^k (q;q)_{i_j}  (q;q)_{\sum_{s=1}^j i_s} }
\end{equation*} 
\end{corollary} 
 This gives a natural generalization of the following well-known false theta function identity (The Lost Notebook and Other
Unpublished Papers; page 200 in \cite{Ramanujan}) : 
 
\begin{eqnarray*}
(q;q)^2_{\infty}
\sum_{i=0}^{\infty}\frac{q^{i}}{(q;q)_{i}^2}=(q;q)_{\infty}
\sum_{i=0}^{\infty}\frac{q^{i^2+i}}{(q;q)_{i}^2}.
\end{eqnarray*}

\section{Alternating Knots and Rogers-Ramanujan Type Identities }
\label{sec3}
In this section, we review the Rogers-Ramanujan type identities that were recovered in the literature using techniques related to the tail of the colored Jones polynomial of alternating links. Furthermore, we show the false theta function type identities that we recover in this paper from Pretzel knots.

The general two variable Ramanujan false theta function is given by (e.g. \cite{Andrews}):
\begin{equation}
\Psi(a,b)=\sum\limits_{i=0}^{\infty}a^{\frac{i(i+1)}{2}}b^{\frac{i(i-1)}{2}}-\sum\limits_{i=1}^{\infty}a^{\frac{i(i-1)}{2}}b^{\frac{i(i+1)}{2}}
\end{equation}

When $a=q^3$ and $b=q$, we obtain the following well-known identities:

\begin{eqnarray}
\Psi(q^3,q)=\sum_{k=0}^{\infty} (-1)^kq^{\frac{k^2+k}{2}}
\label{falseRam}
=(q;q)_{\infty}
\sum_{k=0}^{\infty}\frac{q^{k^2+k}}{(q;q)^2_{k}}
\label{falseRam2}=(q;q)^2_{\infty}
\sum_{k=0}^{\infty}\frac{q^{k}}{(q;q)^2_{k}}
\end{eqnarray}

In \cite{Hajij1}, the second author recovered the second identity in \ref{falseRam} using the tail of the $(2,4)$-torus link. Furthermore, the tail of the colored Jones polynomial of $(2,2k)$-torus links, where $k\geq 2$, to give a natural extension of the same identity \ref{falseRam}. For all $k\geq 2$, this identity is given by:

\begin{eqnarray}
\label{fock2}
\Psi(q^{2k-1},q)=(q;q)_{\infty}\sum\limits_{l_1=0}^{\infty}\sum\limits_{l_{2}=0}^{\infty}...\sum\limits_{l_{k-1}=0}^{\infty}\frac{q^{\sum\limits_{j=1}^{k-1}(i_j(i_j+1))}}{(q;q)^2_{l_{k-1}}\prod\limits_{j=1}^{k-2}(q;q)_{l_j}}
\end{eqnarray}
where $i_j=\sum\limits_{s=j}^{k-1}l_s$. On the other hand, a similar identity for the theta function, known as Roger-Ramanujan identity for the two-variable theta function, can be recovered from the tail of the colored Jones polynomial of $(2,2k+1)$-torus knots.  Recall that the general two variable Ramanujan theta function is defined by:
\begin{equation}
\label{oneone}
f(a,b)=\sum\limits_{i=0}^{\infty}a^{i(i+1)/2}b^{i(i-1)/2}+\sum\limits_{i=1}^{\infty}a^{i(i-1)/2}b^{i(i+1)/2}
\end{equation}

The function $f(a,b)$ specializes to: 
\begin{equation}
f(-q^{2k},-q)=\sum\limits_{i=0}^{\infty}(-1)^i q^{k(i^2+i)}q^{i(i-1)/2}+\sum\limits_{i=1}^{\infty}(-1)^i q^{k(i^2-i)}q^{i(i+1)/2}.
\end{equation}
  For $k \geq 1$ the Roger-Ramanujan identity for the theta function is given by:
\begin{eqnarray}
\label{and1}
f(-q^{2k},-q)=(q;q)_{\infty}\sum\limits_{l_1=0}^{\infty}\sum\limits_{l_{2}=0}^{\infty}...\sum\limits_{l_{k-1}=0}^{\infty}\frac{q^{\sum\limits_{j=1}^{k-1}(i_j(i_j+1))}}{\prod\limits_{j=1}^{k-1}(q;q)_{l_j}}
\end{eqnarray}
where $i_j=\sum\limits_{s=j}^{k-1}l_s$. The identities \ref{fock2} and \ref{and1} were recovered using a unified skein theoretic method in \cite{Hajij1}. Note that the identities (\ref{and1}) and (\ref{fock2}) are coming from cyclic graphs with odd and even number of vertices respectively. It is plausible to think that a natural family of knots, or graphs, gives rise to a natural family of $q$-series identities. In this paper we recover the third identity (\ref{falseRam2}) using the tail of the colored Jones polynomial. This $q$-series correspond to the graph given in the following Figure. 
\begin{figure}[H]
  \centering
 {\includegraphics[scale=0.15]{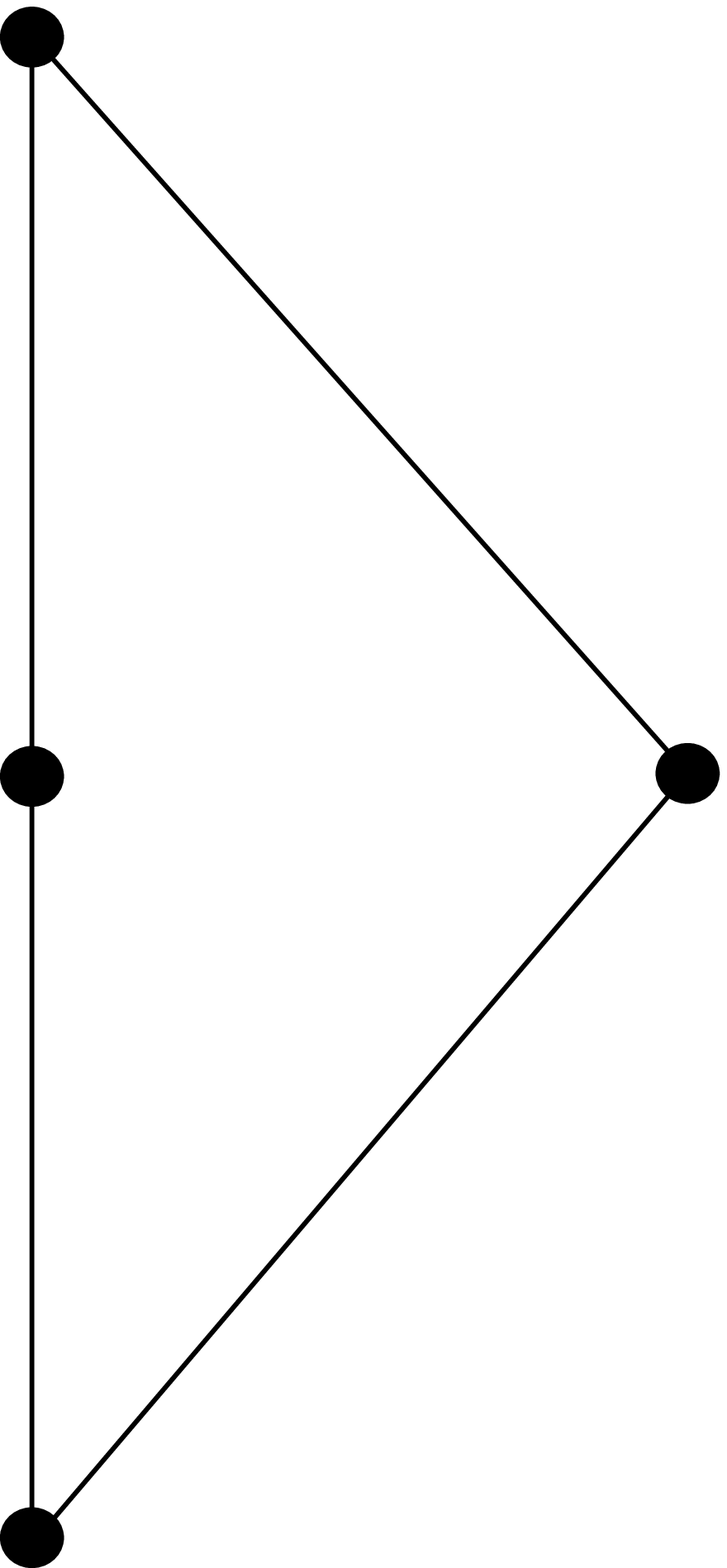}
  \label{pretzel_simple}}
\end{figure}

 Furthermore, we give a natural generalization of this identity using the tail of the graph $L_{k}$, where $k \geq 1$, given in Figure \ref{pretzel_2}. Note that the graph $L_{k-1}$ corresponds to the the pretzel knot $P_k$ in Proposition \ref{side1}.  

\begin{figure}[H]
  \centering
 {\includegraphics[scale=0.2]{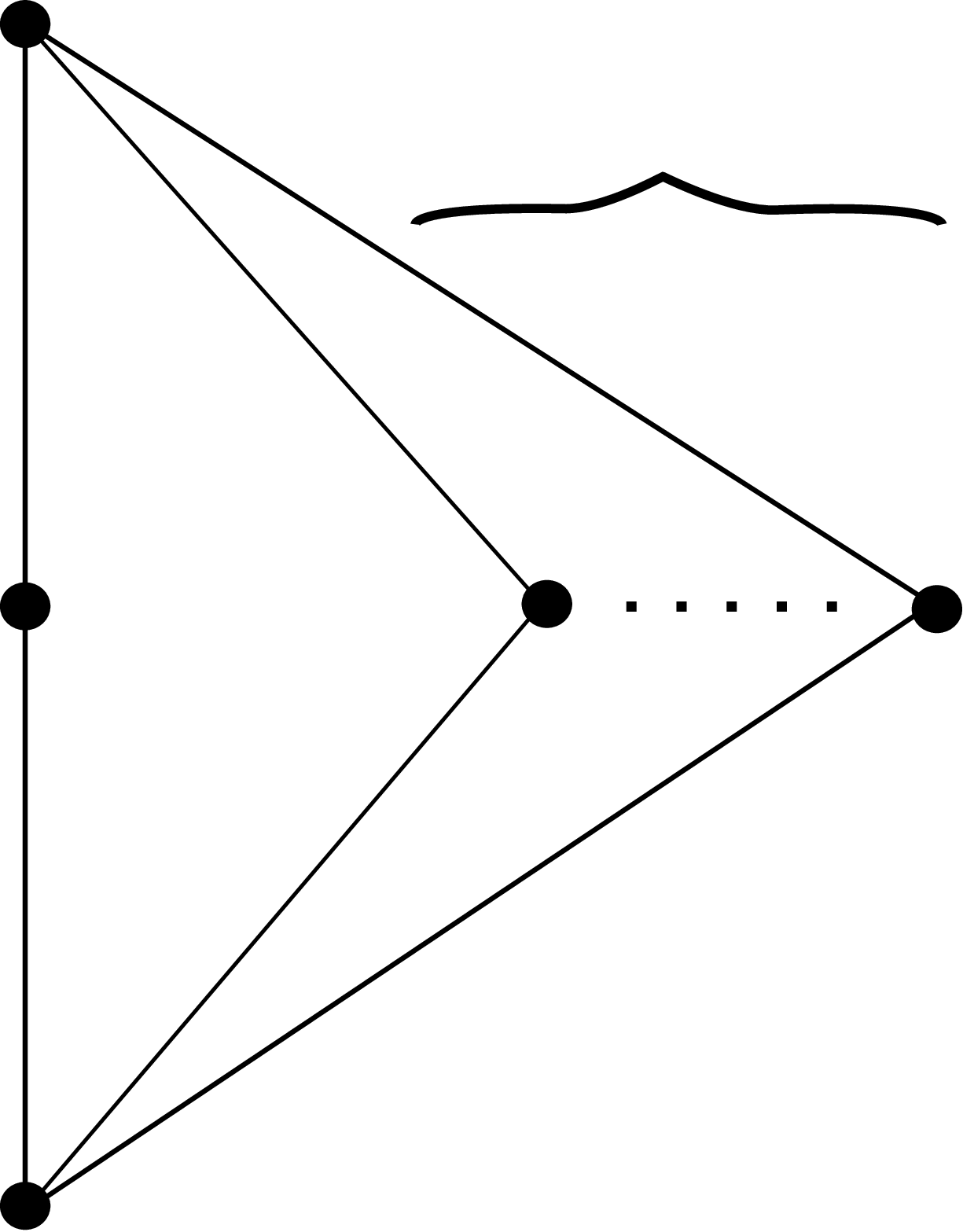}
 \put(-40,98){$k$ vertices}
  \caption{The graph $L_{k}$}
  \label{pretzel_2}}
\end{figure} 
We show that this generalization is given by: 
\begin{equation}
\label{new}
(q;q)_{\infty} \sum_{i=0}^{\infty}  \frac{q^{i}}{(q;q)_i^{k+1}} = \sum_{i_1=0}^{\infty}...\sum_{i_k=0}^{\infty}\frac{q^{\sum_{j=1}^k i_j+i_j^2+\sum_{s=2}^k \sum_{j=s}^{k} i_{s-1}i_j}}{\prod_{j=1}^k (q;q)_{i_j}  (q;q)_{\sum_{s=1}^j i_s} }
\end{equation}

\section{The Tail of the Colored Jones Polynomial of the Pretzel Knots $P(2k+1,2,2u+1)$}\label{sec4}
In \cite{Hajij1}, the second author computed the tail of the knot $8_5$. The tail of this knot is given by: 
\begin{equation}
T_{8_5}(q)=(q;q)^2_\infty\sum\limits_{i=0}^{\infty}\sum\limits_{j=0}^{\infty}\frac{ q^{(i +i^2+j+j^2)}(q;q)_{i+j}}{(q;q)^2_i(q;q)^2_j}.
\end{equation}
The series $T_{8_5}$ is similar to the following $q$-series:  
\begin{equation}
\label{Gamma}
T_{\Gamma}=(\Psi(q^3,q))^2=(q;q)^2_\infty\sum\limits_{i=0}^{\infty}\sum\limits_{j=0}^{\infty}\frac{ q^{(i +i^2+j+j^2)}}{(q;q)^2_i(q;q)^2_j}
\end{equation}
where $\Gamma$ is the graph shown on the right handside of Figure \ref{graph}. This similarity is not surprising since the graph associated to the knot $8_5$ is given in left handside of the Figure \ref{graph}. 
\begin{figure}[H]
  \centering
 {\includegraphics[scale=0.04]{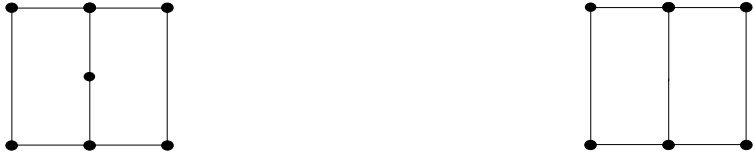}
  \caption{The reduced $B$-graph for $8_5$ on the left and the graph $\Gamma$ on the right. }
  \label{graph}}
\end{figure}
Motivated by this observation, in this section we will study the tail of the family of graphs given in Figure \ref{graph3} and show the relation between this $q$-series and the false theta function. Note that this graph corresponds to pretzel knots $(2k+1,2,2u+1)$ where $u,k \geq 1$.

\begin{figure}[H]
  \centering
 {\includegraphics[scale=0.04]{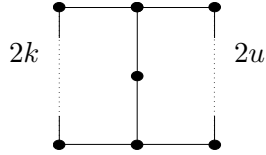}
  \put(-80,33){$2k$ }
  \put(5,33){$2u$ }
  \caption{The graph $\Phi_{k,u}$}
  \label{graph3}}
\end{figure}
For our tail computations, we will study the element  $\left(
   \begin{minipage}[h]{0.1\linewidth}
        \hspace{1pt}
        \scalebox{0.3}{\includegraphics{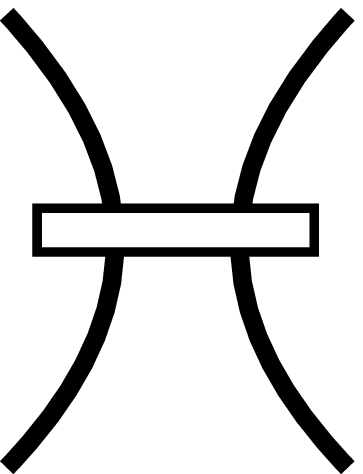}}
             \put(-35,+31){\footnotesize{$n$}}
             \put(-35,5){\footnotesize{$n$}}
             \put(-1,+31){\footnotesize{$n$}}
             \put(-1,5){\footnotesize{$n$}}
   \end{minipage} \hspace{-5pt} \right)^{\otimes t}$ for $t \geq 1$. Note that when $t=2$ we obtain the bubble element $\mathcal{B}^{n,n}_{n,n}(n,n)$.

\begin{lemma}
\label{mainlemma}
Let $n\geq 1$, then we have
\begin{enumerate}
\item For $k\geq 1$, we have
\small{
 \begin{eqnarray*}
 \left(
   \begin{minipage}[h]{0.1\linewidth}
        \hspace{1pt}
        \scalebox{0.3}{\includegraphics{halfbubble}}
             \put(-35,+28){\footnotesize{$n$}}
             \put(-1,+28){\footnotesize{$n$}}
   \end{minipage} \hspace{0pt} \right)^{\otimes (2k+1)}=\sum\limits_{i_1=0}^{n}\sum\limits_{i_2=0}^{i_1}...\sum\limits_{i_k=0}^{i_{k-1}}E_{n,i_1,...,i_k}  
   \begin{minipage}[h]{0.19\linewidth}
        \vspace{0pt}
        \scalebox{0.260}{\includegraphics{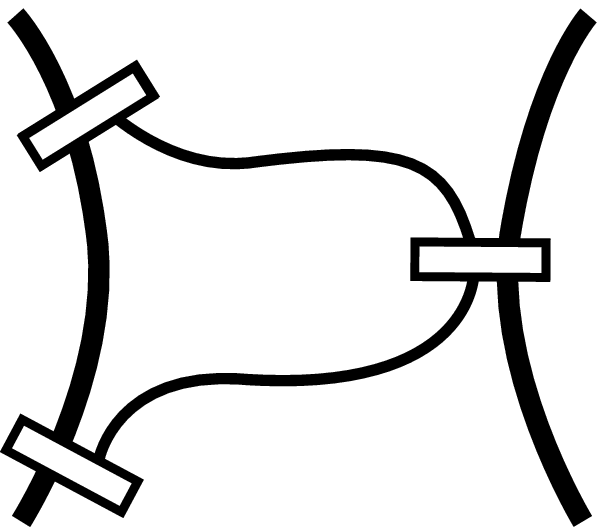}}
         \put(-50,35){$n$}
          \put(-30,35){$i_k$}
           \put(-30,3){$i_k$}
          \put(2,35){$n$}
           \end{minipage}
  \end{eqnarray*}
}
where 
\begin{equation}
E_{n,i_1,...,i_k}=
\left\lceil 
\begin{array}{cc}
n & n \\ 
n & n%
\end{array}%
\right\rceil _{i_1}
\frac{\Delta_{2n}}{\Delta_{n+i_1}}
\prod_{j=2}^k
\left\lceil 
\begin{array}{cc}
n & i_{j-1} \\ 
n & n%
\end{array}%
\right\rceil _{i_{j}}
\frac{\Delta_{2n}}{\Delta_{n+i_j}}
\end{equation}

\item For $k\geq 2$, we have  \small{
 \begin{eqnarray*}
  \left(
   \begin{minipage}[h]{0.1\linewidth}
        \hspace{1pt}
        \scalebox{0.3}{\includegraphics{halfbubble}}
             \put(-35,+28){\footnotesize{$n$}}
             \put(-1,+28){\footnotesize{$n$}}
   \end{minipage} \hspace{0pt} \right)^{\otimes (2k)}=\sum\limits_{i_1=0}^{n}\sum\limits_{i_2=0}^{i_1}...\sum\limits_{i_k=0}^{i_{k-1}}P_{n,i_1,...,i_k}\begin{minipage}[h]{0.19\linewidth}
        \vspace{0pt}
        \scalebox{0.260}{\includegraphics{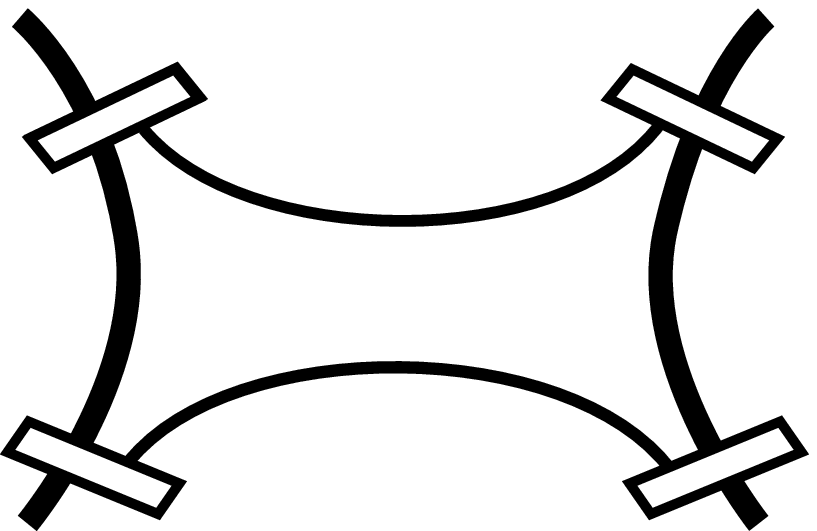}}
         \put(-57,38){$n$}
          \put(-37,33){$i_{k}$}
           \put(-37,3){$i_{k}$}
          \put(2,35){$n$}
           \end{minipage}.
  \end{eqnarray*}
} 

\end{enumerate}

where

\begin{equation}
P_{n,i_1,...,i_k}
=\left\lceil 
\begin{array}{cc}
n & n \\ 
n & n%
\end{array}%
\right\rceil _{i_1}
\frac{\Delta_{2n}}{\Delta_{n+i_1}}
\prod_{j=2}^{k-1}
\left\lceil 
\begin{array}{cc}
n & i_{j-1} \\ 
n & n%
\end{array}%
\right\rceil _{i_{j}}\
\frac{\Delta_{2n}}{\Delta_{n+i_j}}\left\lceil 
\begin{array}{cc}
n & i_{k-1} \\ 
n & n%
\end{array}%
\right\rceil _{i_{k}}  
\end{equation}

\end{lemma}

\begin{proof}
(1) Note first that

\small{
 \begin{eqnarray*}
 \left(
   \begin{minipage}[h]{0.1\linewidth}
        \hspace{1pt}
        \scalebox{0.3}{\includegraphics{halfbubble}}
             \put(-35,+28){\footnotesize{$n$}}
             \put(-1,+28){\footnotesize{$n$}}
   \end{minipage} \hspace{0pt} \right)^{\otimes (2k+1)}=\hspace{15pt}
   \begin{minipage}[h]{0.28\linewidth}
         \vspace{0pt}
         \scalebox{0.265}{\includegraphics{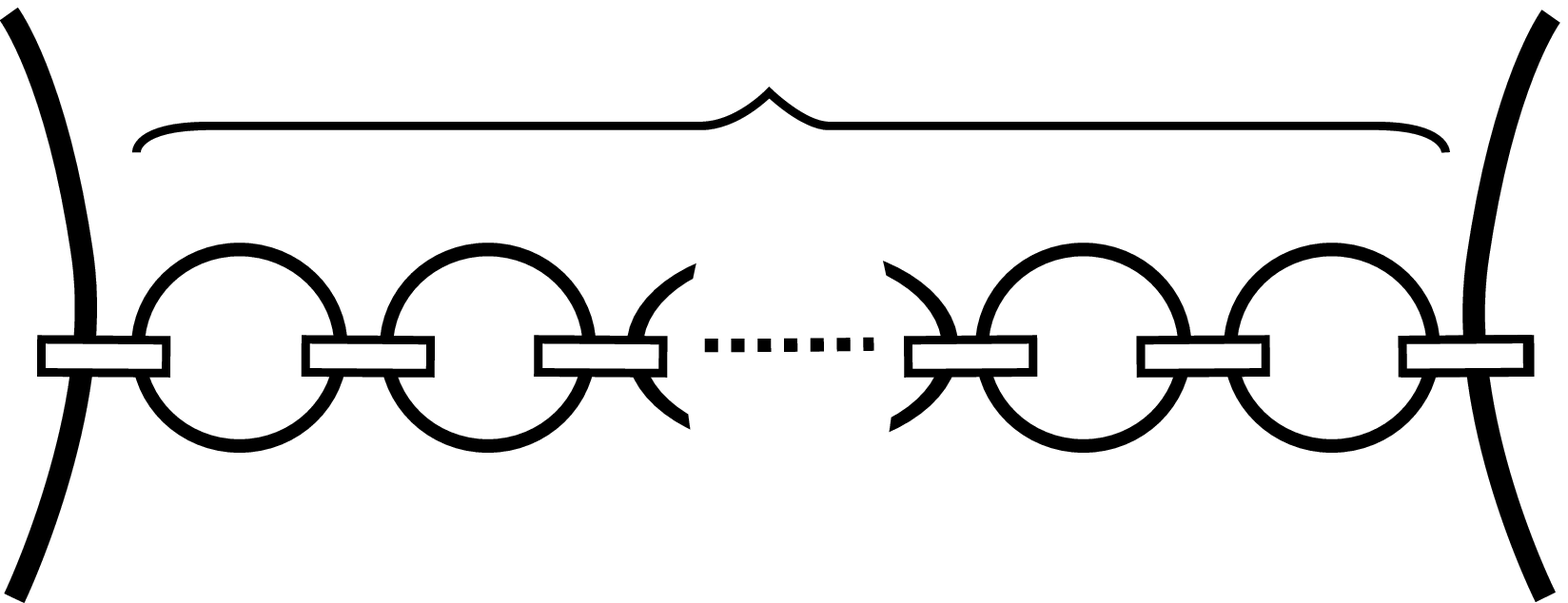}}
        \put(-75,+47){$2k$}
         \put(-86,+7){$n$}
         \put(-108,+7){$n$}
             \put(-23,+7){$n$}
           \put(-40,+7){$n$}
             \put(-86,+30){$n$}
         \put(-108,+30){$n$}
             \put(-23,+30){$n$}
           \put(-40,+30){$n$}
         \put(0,40){$n$}
         \put(-135,40){$n$} 
   \end{minipage}
\end{eqnarray*}}
   
We apply the bubble expansion formula $k$ times on the previous equation to obtain:
\small{
 \begin{eqnarray*}
 \left(
   \begin{minipage}[h]{0.1\linewidth}
        \hspace{1pt}
        \scalebox{0.3}{\includegraphics{halfbubble}}
             \put(-35,+28){\footnotesize{$n$}}
             \put(-1,+28){\footnotesize{$n$}}
   \end{minipage} \hspace{0pt} \right)^{\otimes (2k+1)}&=&\hspace{15pt}
   \begin{minipage}[h]{0.28\linewidth}
         \vspace{0pt}
         \scalebox{0.265}{\includegraphics{gn_bubbles}}
        \put(-75,+47){$2k$}
         \put(-86,+7){$n$}
         \put(-108,+7){$n$}
             \put(-23,+7){$n$}
           \put(-40,+7){$n$}
             \put(-86,+30){$n$}
         \put(-108,+30){$n$}
             \put(-23,+30){$n$}
           \put(-40,+30){$n$}
         \put(0,40){$n$}
         \put(-135,40){$n$} 
   \end{minipage}\\&=&\sum\limits_{i_1=0}^{n}\sum\limits_{i_2=0}^{i_1}...\sum\limits_{i_k=0}^{i_{k-1}}\left\lceil 
\begin{array}{cc}
n & n \\ 
n & n%
\end{array}%
\right\rceil _{i_1}
\frac{\Delta_{2n}}{\Delta_{n+i_1}}
\prod_{j=2}^k
\left\lceil 
\begin{array}{cc}
n & i_{j-1} \\ 
n & n%
\end{array}%
\right\rceil _{i_{j}}\\&\times&
\frac{\Delta_{2n}}{\Delta_{n+i_j}}  
   \begin{minipage}[h]{0.2\linewidth}
        \hspace{5pt}
        \scalebox{0.260}{\includegraphics{gn_bubbles_4}}
         \put(-50,35){$n$}
          \put(-30,35){$i_k$}
           \put(-30,3){$i_k$}
          \put(2,35){$n$}
           \end{minipage}
  \end{eqnarray*}
}
the result then follows.\\
(2)
We apply the bubble expansion formula $k-1$ times and we obtain:

\small{
 \begin{eqnarray*}
  \left(
   \begin{minipage}[h]{0.1\linewidth}
        \hspace{1pt}
        \scalebox{0.3}{\includegraphics{halfbubble}}
             \put(-35,+28){\footnotesize{$n$}}
             \put(-1,+28){\footnotesize{$n$}}
   \end{minipage} \hspace{0pt} \right)^{\otimes (2k)}&=&\hspace{15pt}
   \begin{minipage}[h]{0.28\linewidth}
         \vspace{0pt}
         \scalebox{0.265}{\includegraphics{gn_bubbles}}
        \put(-75,+47){$2k-1$}
         \put(-86,+7){$n$}
         \put(-108,+7){$n$}
             \put(-23,+7){$n$}
           \put(-40,+7){$n$}
             \put(-86,+30){$n$}
         \put(-108,+30){$n$}
             \put(-23,+30){$n$}
           \put(-40,+30){$n$}
         \put(0,40){$n$}
         \put(-135,40){$n$} 
   \end{minipage}\\&=&\sum\limits_{i_1=0}^{n}\sum\limits_{i_2=0}^{i_1}...\sum\limits_{i_k=0}^{i_{k-1}}\left\lceil 
\begin{array}{cc}
n & n \\ 
n & n%
\end{array}%
\right\rceil _{i_1}
\frac{\Delta_{2n}}{\Delta_{n+i_1}}
\prod_{j=2}^{k-1}
\left\lceil 
\begin{array}{cc}
n & i_{j-1} \\ 
n & n%
\end{array}%
\right\rceil _{i_{j}}\\&\times&
\frac{\Delta_{2n}}{\Delta_{n+i_j}}\left\lceil 
\begin{array}{cc}
n & i_{k-1} \\ 
n & n%
\end{array}%
\right\rceil _{i_{k}}  
   \begin{minipage}[h]{0.19\linewidth}
        \vspace{0pt}
        \scalebox{0.260}{\includegraphics{gn_bubbles_5}}
         \put(-57,38){$n$}
          \put(-37,33){$i_{k}$}
           \put(-37,3){$i_{k}$}
          \put(2,35){$n$}
           \end{minipage}.
  \end{eqnarray*}
}

\end{proof}

\begin{lemma}
\label{technical lemma}
\begin{enumerate}
\item Let $k\geq 1$. Then,
 \small{
 \begin{eqnarray*}
 \sum\limits_{i_1=0}^{n}\sum\limits_{i_2=0}^{i_1}...\sum\limits_{i_k=0}^{i_{k-1}}E_{n,i_1,...,i_k}\begin{minipage}[h]{0.19\linewidth}
        \vspace{0pt}
        \scalebox{0.260}{\includegraphics{gn_bubbles_4}}
         \put(-50,35){$n$}
          \put(-30,35){$i_k$}
           \put(-30,3){$i_k$}
          \put(2,35){$n$}
           \end{minipage}= \sum\limits_{i_k=0}^{n}\sum\limits_{i_{k-1}=i_k}^{n}...\sum\limits_{i_1=i_2}^{n}E_{n,i_1,...,i_k}\begin{minipage}[h]{0.19\linewidth}
        \vspace{0pt}
        \scalebox{0.260}{\includegraphics{gn_bubbles_4}}
        \put(-50,35){$n$}
          \put(-30,35){$i_k$}
           \put(-30,3){$i_k$}
          \put(2,35){$n$}
           \end{minipage}.
  \end{eqnarray*}
}
Moreover,
\begin{eqnarray*}
E_{n,i_1,..,i_k}&=&(-1)^{k n+\sum\limits_{j=1}^{k} i_j }q^{ k n/2+ \sum\limits_{j=1}^{k}(i_j(i_j/2+1))}\\ &\times&
 \frac{(q;q)^{4k+2}_{n}(q;q)_{3n-i_1+1}}{(q;q)^{k+1}_{2n} (q;q)_{2n+1}(q;q)_{n-i_1}(q;q)^2_{n-i_k}(q;q)^2_{i_k}}
\\ &\prod\limits_{j=2}^{k}&\frac{(q;q)_{i_{j-1}-i_j +2n+1}}{(q;q)_{i_{j-1}-i_j}(q;q)_{n+i_{j-1}}(q;q)^2_{n-i_{j-1}}(q;q)_{n+i_{j-1}+1}}\prod\limits_{j=1}^{k}\frac{\Delta_{2n}}{\Delta_{n+i_j}}
\end{eqnarray*}

\item For $k\geq1$, we have
 \small{
 \begin{eqnarray*}
 \sum\limits_{i_1=0}^{n}\sum\limits_{i_2=0}^{i_1}...\sum\limits_{i_k=0}^{i_{k-1}}P_{n,i_1,...,i_k}\begin{minipage}[h]{0.19\linewidth}
        \vspace{0pt}
        \scalebox{0.260}{\includegraphics{gn_bubbles_5}}
         \put(-57,38){$n$}
          \put(-37,33){$i_{k}$}
           \put(-37,3){$i_{k}$}
          \put(2,35){$n$}
           \end{minipage}= \sum\limits_{i_k=0}^{n}\sum\limits_{i_{k-1}=i_k}^{n}...\sum\limits_{i_1=i_2}^{n}P_{n,i_1,...,i_k}\begin{minipage}[h]{0.19\linewidth}
        \vspace{0pt}
        \scalebox{0.260}{\includegraphics{gn_bubbles_5}}
         \put(-57,38){$n$}
          \put(-37,33){$i_{k}$}
           \put(-37,3){$i_{k}$}
          \put(2,35){$n$}
           \end{minipage}.
  \end{eqnarray*}
}
Moreover,
\begin{eqnarray*}
P_{n,i_1,..,i_k}&=&(-1)^{k n+\sum\limits_{j=1}^{k} i_j }q^{ k n/2+ \sum\limits_{j=1}^{k}(i_j(i_j/2+1))}\\ &\times&
 \frac{(q;q)^{4k+2}_{n}(q;q)_{3n-i_1+1}}{(q;q)^{k+1}_{2n} (q;q)_{2n+1}(q;q)_{n-i_1}(q;q)^2_{n-i_k}(q;q)^2_{i_k}}
\\ &\prod\limits_{j=2}^{k}&\frac{(q;q)_{i_{j-1}-i_j +2n+1}}{(q;q)_{i_{j-1}-i_j}(q;q)_{n+i_{j-1}}(q;q)^2_{n-i_{j-1}}(q;q)_{n+i_{j-1}+1}}\prod\limits_{j=1}^{k-1}\frac{\Delta_{2n}}{\Delta_{n+i_j}}
\end{eqnarray*}
\end{enumerate} 
\end{lemma}
\begin{proof}
\begin{enumerate}
\item
Using the fact that 

\begin{equation*}
\prod\limits_{i=0}^{j}[n-i]=q^{(2 + 3 j + j^2 - 2 n - 2 j n)/4} (1 - q)^{-1 - j}\frac{(q;q)_n}{(q;q)_{n-j-1}}
\end{equation*}

one obtains :

\begin{equation*}
\label{eq}
\left\lceil 
\begin{array}{cc}
n & a \\ 
n & n%
\end{array}%
\right\rceil _{b}=(-1)^{n+b}q^{b/2+b^2-n/2} 
\frac{(q; q)^2_a(q;q)^4_{n}(q;q)_{1+a-b+2n}}{(q;q)_{a-b}(q ;q)^2_{b}(q ;q)_{2n}(q ;q)_{a+n}(q ;q)_{1+a+n}(q ;q)^2_{-b+n}}.
\end{equation*}
This implies,

\begin{eqnarray*}
\left\lceil 
\begin{array}{cc}
n & n \\ 
n & n%
\end{array}%
\right\rceil _{i_1}
\prod\limits_{j=2}^{k}
\left\lceil 
\begin{array}{cc}
n & i_{j-1} \\ 
n & n%
\end{array}%
\right\rceil _{i_j}
&=&(-1)^{k n+\sum\limits_{j=1}^{k} i_j }q^{ k n/2+ \sum\limits_{j=1}^{k}(i_j(i_j/2+1))}\\ &\times&
 \frac{(q;q)^{4k+2}_{n}(q;q)_{3n-i_1+1}}{(q;q)^{k+1}_{2n} (q;q)_{2n+1}(q;q)_{n-i_1}(q;q)^2_{n-i_k}(q;q)^2_{i_k}}
\\ &\prod\limits_{j=2}^{k}&\frac{(q;q)_{i_{j-1}-i_j +2n+1}}{(q;q)_{i_{j-1}-i_j}(q;q)_{n+i_{j-1}}(q;q)^2_{n-i_{j-1}}(q;q)_{n+i_{j-1}+1}}  
\end{eqnarray*}
On the other hand, one has
\begin{equation*}
\sum\limits_{i_1=0}^{n}\sum\limits_{i_2=0}^{i_1}...\sum\limits_{i_k=0}^{i_{k-1}}F(i_1,...,i_k)=\sum\limits_{i_k=0}^{n}\sum\limits_{i_{k-1}=i_k}^{n}...\sum\limits_{i_1=i_2}^{n}F(i_1,...,i_k)
\end{equation*}
The result then follows.
\item The proof is similar to $(1)$.
\end{enumerate}
\end{proof}

\begin{theorem}
\label{technical}
The tail of the graph $\Phi_{k,u}$ is given by
\begin{equation*}
T_{\Phi_{k,u}}(q)=(q;q)^2_\infty\sum\limits_{l_1=0}^{\infty}...\sum\limits_{l_k=0}^{\infty}\sum\limits_{p_1=0}^{\infty}...\sum\limits_{p_u=0}^{\infty}g(q;l_1,...,l_k)g(q;p_1,...,p_u)(q;q)_{i+j}
\end{equation*}
where 
\begin{equation*}
g(q;l_1,...,l_k)=\frac{q^{\sum\limits_{j=1}^{k}(i_j(i_j+1))}}{(q;q)^2_{l_{k}}\prod\limits_{j=1}^{k-1}(q;q)_{l_j}}
\end{equation*}
with $i_j=\sum\limits_{s=j}^{k}l_s$.
\end{theorem}
\begin{proof}
Using Theorem \ref{cody thm}, we have

\begin{equation}
T_{\Phi_{k,u}}(q)\doteq_n \frac{ S_B^{(n)}(\Phi_{k,u})}{\Delta_{n}}\bigg|_{A=q^
{1/4}}
\end{equation} 
where $S_B^{(n)}(\Phi_{k,u})$ is the skein element given in Figure \ref{mainskein}.

\begin{figure}[h]
  \centering
 {\includegraphics[scale=0.1]{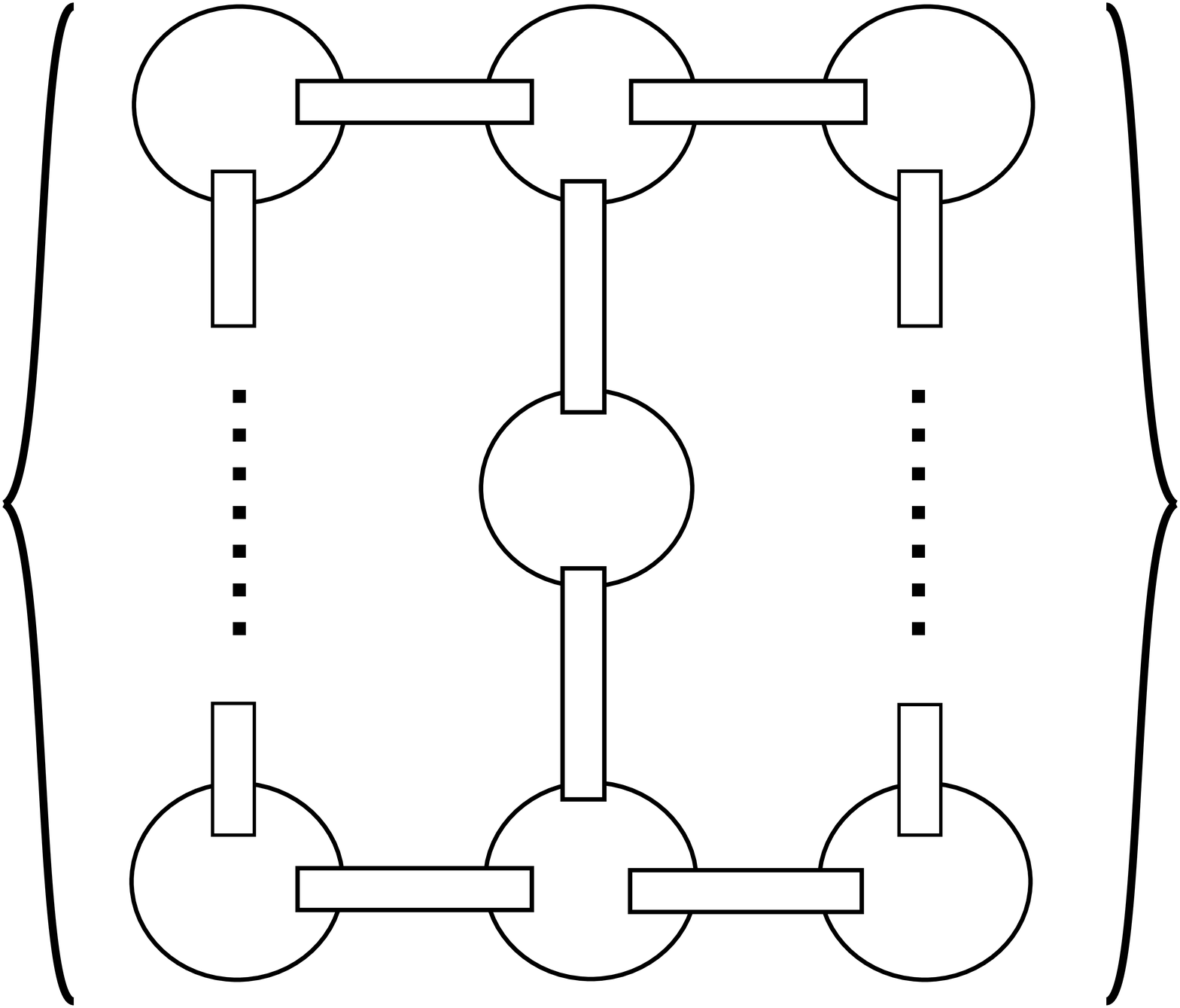}
 \put(-75,43){$n$ }
 \put(-70,73){$n$ }
 \put(-65,95){$n$ }
  \put(-25,95){$n$ }
  \put(-25,-5){$n$}
  \put(-60,-5){$n$ }
  \put(-90,-5){$n$}
  \put(-90,95){$n$}
  \put(-81,73){$n$}
  \put(-47,73){$n$}
  \put(-37,73){$n$}
  \put(-49,19){$n$}
  \put(-67,19){$n$}
  \put(-81,19){$n$}
  \put(-47,37){$n$}
  \put(-37,19){$n$}
  \put(-125,43){$2k$}
  \put(5,43){$2u$ }
  \caption{The skein element $S_B^{(n)}(\Phi_{k,u})$.}
  \label{mainskein}}
\end{figure}

Using Lemma \ref{mainlemma}, we can write
   {\footnotesize
\begin{eqnarray}
\label{kharaa}
  T_{\Phi_{k,u}}(q)
   \doteq_n \frac{1}{\Delta_n} \sum\limits_{i_1=0}^{n}\sum\limits_{i_2=0}^{i_1}...\sum\limits_{i_k=0}^{i_{k-1}} \sum\limits_{j_1=0}^{n}\sum\limits_{j_2=0}^{j_1}...\sum\limits_{j_u=0}^{j_{u-1}} P_{n,i_1,...,i_{k}} P_{n,j_1,...,j_{u}}\frac{\Delta_{2n}}{\Delta{n+i_k}}\frac{\Delta_{2n}}{\Delta_{n+j_u}}  
  \begin{minipage}[h]{0.19\linewidth}
        \hspace{5pt}
        \scalebox{0.1}{\includegraphics{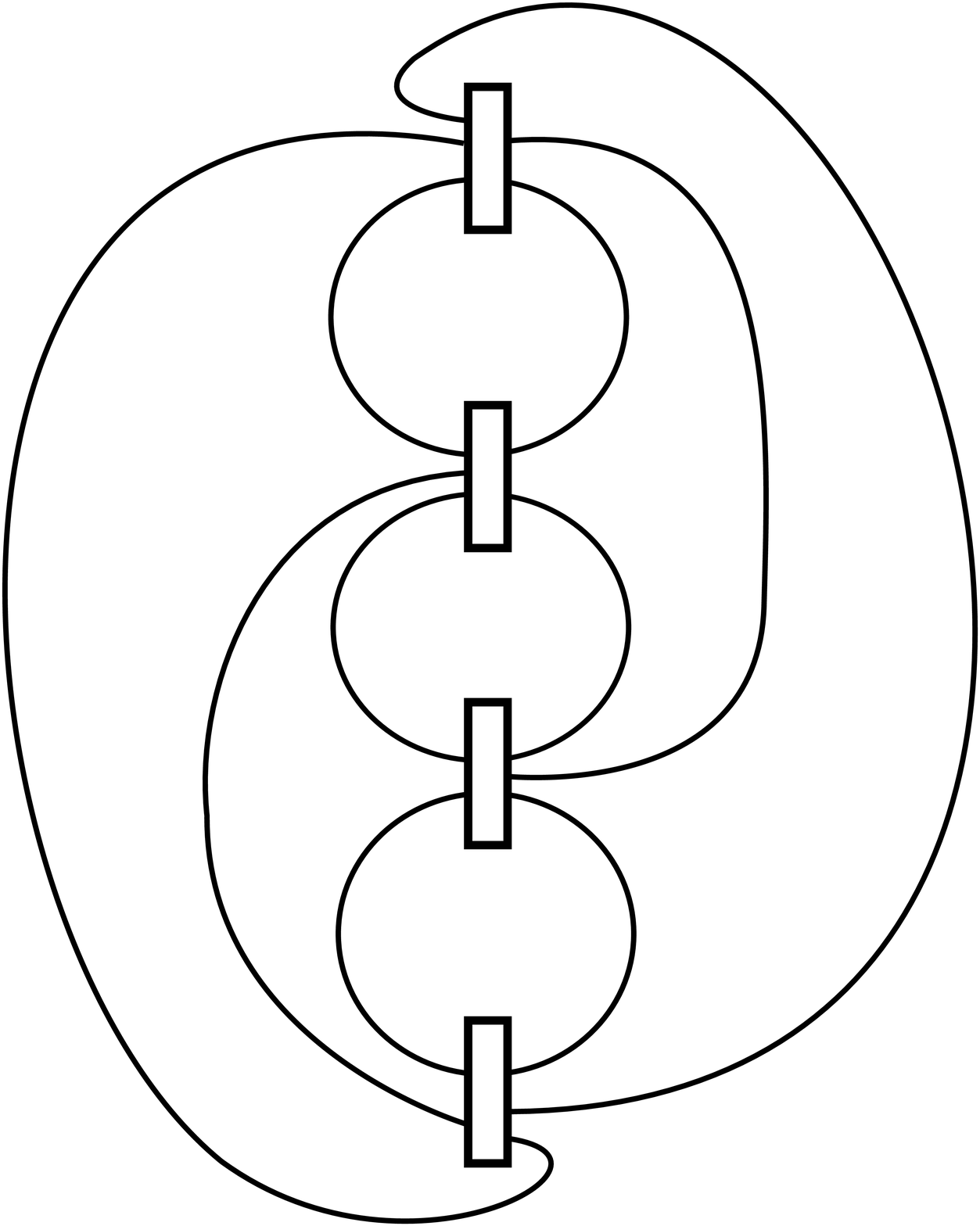}}
         \put(-98,41){$i_k$ }
 		\put(-75,76){$i_k$ }
 		\put(-92,89){$n-i_k$ }
 		\put(-48,89){$n$ }
 		\put(0,41){$j_u$ }
 		\put(-21,76){$j_u$}
 		\put(-35,60){$n$ }
 		\put(-65,55){$n$ }
 		\put(-65,27){$n$ }
 		\put(-35,36){$n-j_u$ }
        \end{minipage}
        \end{eqnarray}}
Denote the element on the right handside of (\ref{kharaa}) by $\Gamma_{n,i_k,j_u}$. Now, Lemma 6.6 in \cite{Hajij1} implies:

\begin{eqnarray}
\label{gammafirst}
  \Gamma_{n,i_k,j_u}
  = \left\lceil 
\begin{array}{cc}
i_k & n \\ 
n & n-j_u%
\end{array}%
\right\rceil _{0}\left\lceil 
\begin{array}{cc}
j_u & n \\ 
n & n-i_k%
\end{array}%
\right\rceil _{0}\left\lceil 
\begin{array}{cc}
j_u & i_k \\ 
n & n%
\end{array}%
\right\rceil _{0} \Delta_{i_k+j_u}
       \end{eqnarray}
Here,       
\begin{equation}
\label{mn1}
\left\lceil 
\begin{array}{cc}
j_u & n \\ 
n & n-i_k%
\end{array}%
\right\rceil _{0}=(-1)^{n-i_k} q^{(i_k - n)/2}  
 \frac{(q; q)_{i_k+j_u}(q;q)_{n}(q;q)_{n+i_k} (q ;q)_{2 n +j_u+ 1}}{ (q;q)_{i_k}(q;q)_{2n}(q;q)_{j_u+n}(q;q)_{n+j_u+i_k+1}}.
\end{equation}
Moreover Lemma 6.1 in \cite{Hajij1} gives,
 \begin{eqnarray*}
\left\lceil 
\begin{array}{cc}
j_u & i_k \\ 
n & n%
\end{array}%
\right\rceil _{0}\Delta_{i_k+j_u} =
\hspace{1 mm}  
    \begin{minipage}[h]{0.09\linewidth}
        \vspace{0pt}
        \scalebox{0.1}{\includegraphics{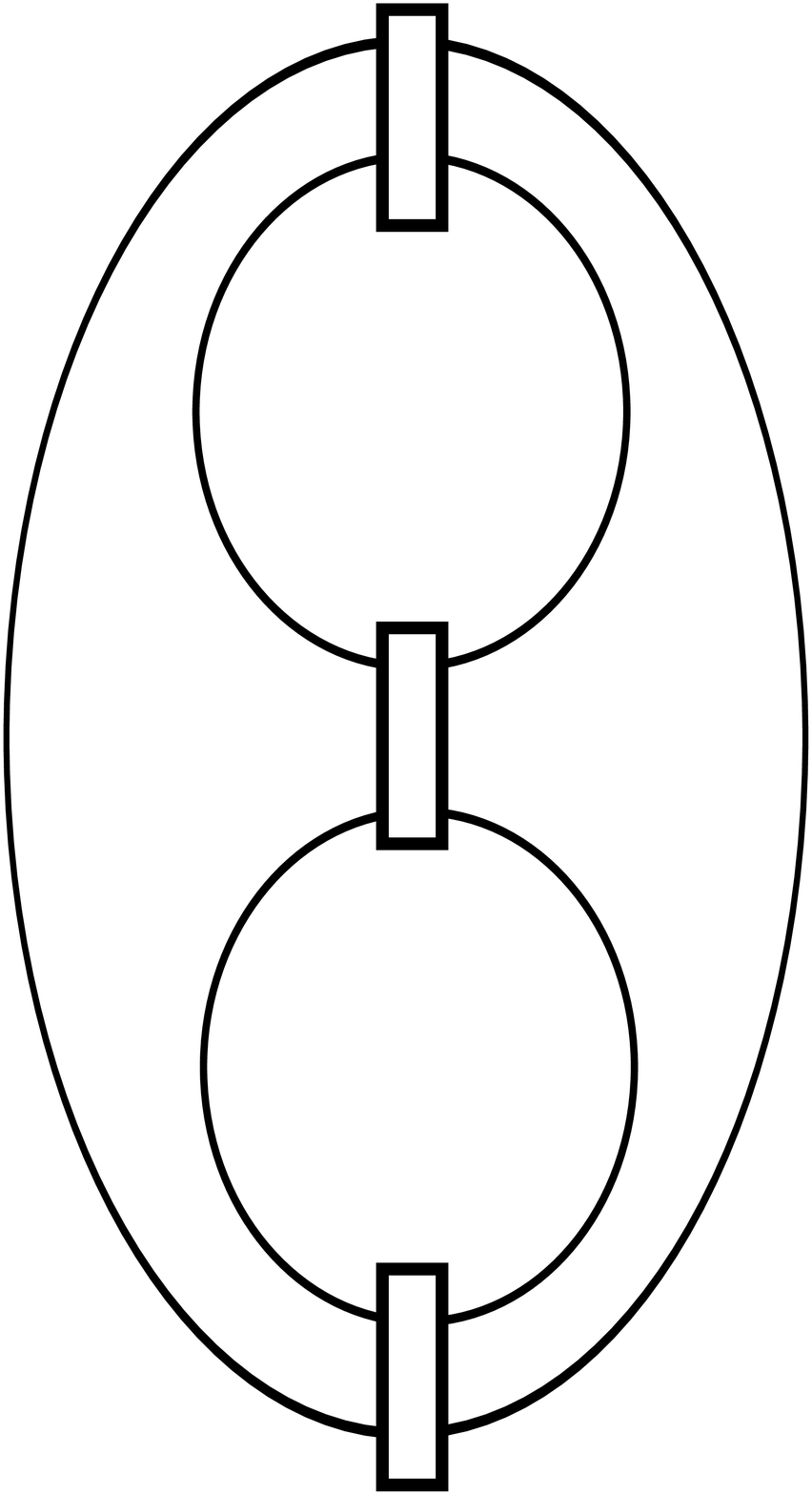}}
             \put(-24,+19){\footnotesize{$i_k$}}
             \put(-24,+69){\footnotesize{$j_u$}}
        \put(-60,+26){\footnotesize{$n$}}
   \end{minipage}
   \end{eqnarray*}

A formula for the skein element on the right hand of the previous equation can be found in \cite{MV}. Using this allows us to obtain:

\begin{equation}
\label{thisthat}
\left\lceil 
\begin{array}{cc}
j_u & i_k \\ 
n & n%
\end{array}%
\right\rceil _{0}\Delta_{i+j}=(-1)^{i_k+j_u+n} q^{-(i_k+j_u+n)/2}  
 \frac{(q;q)_{n}(q;q)_{j_u}(q;q)_{i_k} (q ;q)_{n +j_u+i_k+1}}{ (1-q)(q;q)_{i_k+n}(q;q)_{j_u+n}(q;q)_{j_u+i_k}}.
\end{equation}

Using (\ref{thisthat}) and (\ref{mn1}) in (\ref{gammafirst}) we obtain :

\begin{eqnarray}
\label{mygamma1}
  \Gamma_{n,i_k,j_u}
  = (-1)^{n} q^{-3n/2}
 \frac{(q; q)_{i_k + j_u} (q; q)_n^3 (q; 
  q)_{1 + i_k + 2 n} (q;q)_{1 + j_u + 2 n} }{
(1-q)(q; q)_{2 n}^2 (q;
   q)_{i_k + n}(q; q)_{j_u + n} (q;q)_{ 
  1 + i_k + j_u + n}}
\end{eqnarray}
One the other hand, Lemma \ref{technical lemma} implies 
{\tiny
\begin{equation}
\label{importnat}
\sum\limits_{i_1=0}^{n}...\sum\limits_{i_k=0}^{i_{k-1}} \sum\limits_{j_1=0}^{n}...\sum\limits_{j_u=0}^{j_{u-1}} P_{n,i_{k}} P_{n,j_{u}}\frac{\Delta^2_{2n}}{\Delta_{n+i_k}\Delta_{n+j_u}}\Gamma_{n,i_k,j_u}=\sum\limits_{i_k=0}^n...\sum\limits_{i_1=i_2}^n \sum\limits_{j_u=0}^n...\sum\limits_{j_{1}=j_2}^n P_{n,i_{k}} P_{n,j_{u}}\frac{\Delta^2_{2n}}{\Delta_{n+i_k}\Delta_{n+j_u}}\Gamma_{n,i_k,j_u}  
\end{equation}}

Now

\begin{align}
\label{0}
  \frac{(q;q)_n}{(q;q)_{2n}}&=\frac{\displaystyle\prod_{i=0}^{n-1}(1-q^{i+1})}{\displaystyle\prod_{i=0}^{2n-1}(1-q^{i+1})} \nonumber \\
  &=\frac{1}{\displaystyle\prod_{i=n}^{2n-1}(1-q^{i+1})}\nonumber \\
  &=\displaystyle\prod_{i=0}^{n-1}\frac{1}{(1-q^{i+n+1})}\doteq_n1.
\end{align}

Moreover,
\begin{align}
\label{1}
\frac{(q;q)_{3n-i+1}}{(q;q)_{2n+1}}
&=1-q^{2n+2}+O(2n+3)=_{n}1.
\end{align}
and
\begin{align}
\label{2}
 \frac{(q;q)_{2n+i+1}}{(q;q)_{n+i}}&=\frac{\displaystyle\prod_{k=0}^{3n+i}(1-q^{k+1})}{\displaystyle\prod_{i=0}^{n+i-1}(1-q^{k+1})}\nonumber\\&=\displaystyle\prod_{i=n+i}^{3n+i}(1-q^{k+1})\doteq_n1 .
\end{align}
Hence, using Lemma \ref{technical lemma}, the equation (\ref{mygamma1}) and the facts (\ref{0}), (\ref{1}) and (\ref{2}) in \ref{importnat} yield the equation:

\begin{equation*}
\Phi_{k,u}(q)\doteq_n (q;q)^2_n \sum\limits_{i_k=0}^n...\sum\limits_{i_1=i_2}^n \sum\limits_{j_u=0}^n...\sum\limits_{j_{1}=j_2}^n \frac{q^{\sum\limits_{p=1}^{k}(i_p(i_p+1))}}{(q;q)^2_{i_k}\prod\limits_{p=2}^{k}(q;q)_{i_{p-1}-i_p}}
\frac{q^{\sum\limits_{l=1}^{u}(j_l(j_l+1))}}{(q;q)^2_{j_u}\prod\limits_{l=2}^{u}(q;q)_{j_{l-1}-j_l}}(q;q)_{i_k+j_u}  
\end{equation*}
Now set $s_{p}=i_{p}-i_{p+1}$ for $p=1,...,k-1$ and $s_k=i_k$, we obtain $i_p=\sum\limits_{m=p}^{k}s_m$. Similarly, set $h_{l}=j_{l}-j_{l+1}$ for $l=1,...,u-1$ and $h_u=j_u$, we obtain $j_l=\sum\limits_{r=l}^{u}h_d$. Changing the indexes in the previous equation yield the result.
       
\end{proof}

\section{A Family of Pretzel Knots and Rogers-Ramanunjan Type Identities}\label{sec5}
In this section, the tail of the graph  $L_k$, for $k\geq 1$, given in Figure \ref{pretzel_2} below is computed in two methods. Note that this graph correponds to the pretzel knot $P_{k+1}$ defined in section \ref{sec3}.

\begin{figure}[H]
  \centering
 {\includegraphics[scale=0.2]{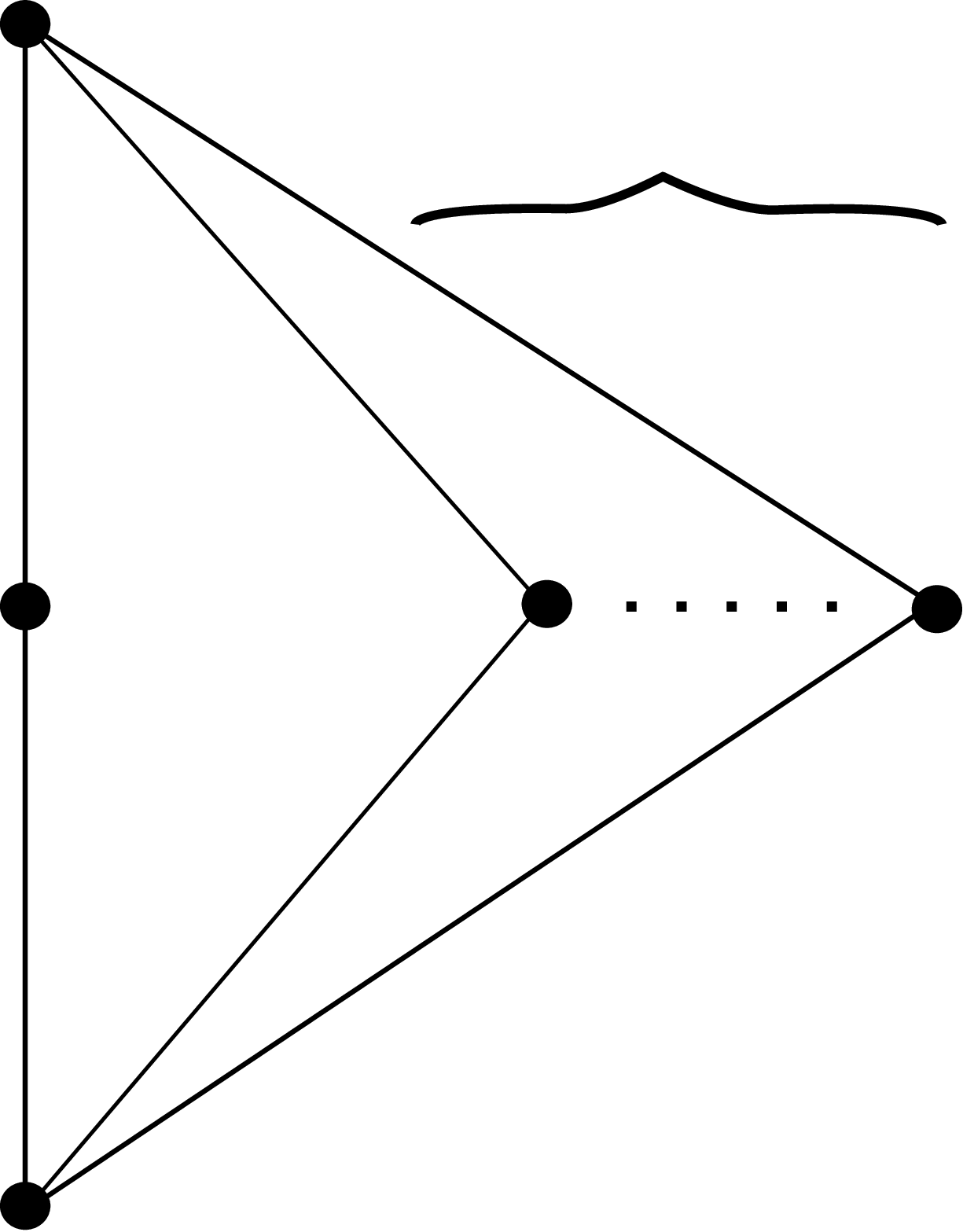}
 \put(-40,98){$k$ vertices}
  \caption{The graph $L_{k}$}
  \label{pretzel_2}}
\end{figure}

The first method utilizes the algorithm given by Masbaum and Vogel in \cite{MV} to compute the evaluation of a quantum spin network in $\mathcal{S}(S^2)$. The second method uses the bubble skein element (\ref{bubble expansion formula121}). Each method give rise to one of side of the $q$-series identities given in (\ref{new}) generalizing the false theta identity given in (\ref{falseRam}). We start first by computing the tail of the graph in Figure \ref{pretzel_2} using the techniques given in \cite{MV}.

\begin{lemma}
\label{main lemma 1}
Let $k,n\geq 1$. Then,

 \begin{eqnarray*}  
  \left( \hspace{10pt}
   \begin{minipage}[h]{0.11\linewidth}
        \vspace{0pt}
        \scalebox{0.22}{\includegraphics{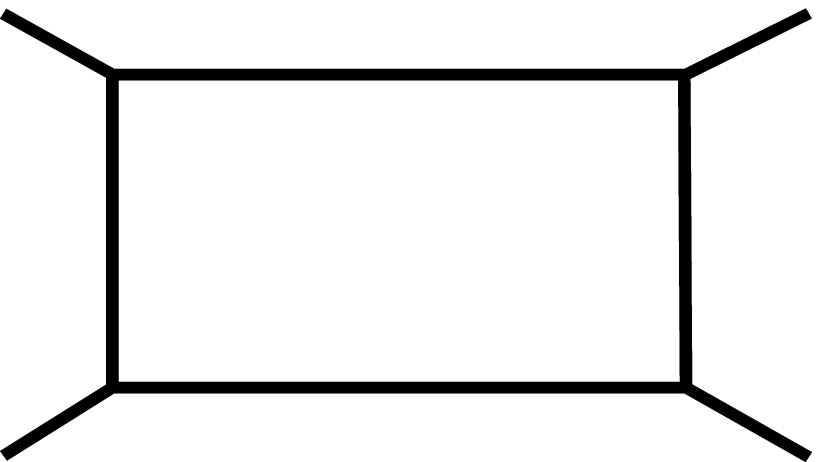}}
             \put(-29,+28){\footnotesize{$2n$}}
             \put(-29,-5){\footnotesize{$2n$}}
        \put(-60,+26){\footnotesize{$n$}}
        \put(-60,-6){\footnotesize{$n$}}
        \put(1,-6){\footnotesize{$n$}}
        \put(1,+26){\footnotesize{$n$}}
   \end{minipage} \hspace{10pt} \right)^{\otimes k}  &=& \sum\limits_{i=0}^{n} C_{n,i}
\hspace{1 mm}  
    \begin{minipage}[h]{0.09\linewidth}
        \vspace{0pt}
        \scalebox{0.36}{\includegraphics{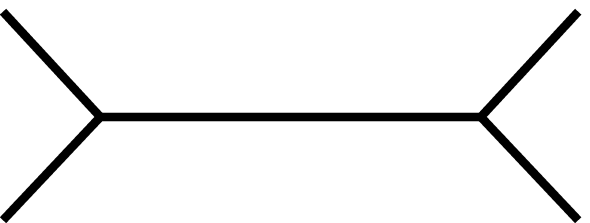}}
             \put(-29,+17){\footnotesize{$2i$}}
        \put(-60,+26){\footnotesize{$n$}}
        \put(-60,-6){\footnotesize{$n$}}
        \put(1,-6){\footnotesize{$n$}}
        \put(1,+26){\footnotesize{$n$}}
   \end{minipage}
   \end{eqnarray*}
where

\begin{equation}
C_{n,i}=\frac{\theta (2n,2n,2i)^{k}  }{\theta (n,n,2i)^{k+1}}\Delta_{2i}
\end{equation}
\end{lemma}

\begin{proof}
Note that
 \begin{eqnarray*}
   \begin{minipage}[h]{0.11\linewidth}
        \vspace{0pt}
        \scalebox{0.22}{\includegraphics{bubble_as_spin}}
             \put(-29,+28){\footnotesize{$2n$}}
             \put(-29,-5){\footnotesize{$2n$}}
        \put(-60,+26){\footnotesize{$n$}}
        \put(-60,-6){\footnotesize{$n$}}
        \put(1,-6){\footnotesize{$n$}}
        \put(1,+26){\footnotesize{$n$}}
   \end{minipage}
   &=& \sum\limits_{i=0}^{n} B_{n,i} 
\hspace{1 mm}  
    \begin{minipage}[h]{0.09\linewidth}
        \vspace{0pt}
        \scalebox{0.36}{\includegraphics{horizantal_basis}}
             \put(-29,+17){\footnotesize{$2i$}}
        \put(-60,+26){\footnotesize{$n$}}
        \put(-60,-6){\footnotesize{$n$}}
        \put(1,-6){\footnotesize{$n$}}
        \put(1,+26){\footnotesize{$n$}}
   \end{minipage}
   \end{eqnarray*}
where

\begin{eqnarray}
B_{n,i}=\frac{ Tet\left[ 
\begin{array}{ccc}
2i & n & n \\ 
n & 2n & 2n%
\end{array}%
\right] Tet\left[ 
\begin{array}{ccc}
2i & 2n & 2n \\ 
n & n & n%
\end{array}%
\right] }{\theta (2n,2n,2i)(\theta (n,n,2i))^2}\Delta _{2i}
\end{eqnarray}
However,

\begin{eqnarray}
\label{22222}
Tet\left[ 
\begin{array}{ccc}
2i & n & n \\ 
n & 2n & 2n%
\end{array}
\right] =Tet\left[ 
\begin{array}{ccc}
2i & 2n & 2n \\ 
n & n & n%
\end{array}%
\right]=\theta(2n,2n,2i).
\end{eqnarray}

Hence
\begin{equation*}
B_{n,i}=\frac{\theta(2n,2n,2i)}{\theta(n,n,2i)^2}\Delta_{2i}
\end{equation*}
Moreover,
\begin{eqnarray*}
    \begin{minipage}[h]{0.14\linewidth}
        \vspace{0pt}
        \scalebox{0.36}{\includegraphics{horizantal_basis}}
             \put(-29,+17){\footnotesize{$2i$}}
        \put(-60,+26){\footnotesize{$n$}}
        \put(-60,-6){\footnotesize{$n$}}
        \put(1,-6){\footnotesize{$n$}}
        \put(1,+26){\footnotesize{$n$}}
   \end{minipage}\otimes \hspace{5pt} \left( \hspace{10pt}
   \begin{minipage}[h]{0.11\linewidth}
        \vspace{0pt}
        \scalebox{0.22}{\includegraphics{bubble_as_spin}}
             \put(-29,+28){\footnotesize{$2n$}}
             \put(-29,-5){\footnotesize{$2n$}}
        \put(-60,+26){\footnotesize{$n$}}
        \put(-60,-6){\footnotesize{$n$}}
        \put(1,-6){\footnotesize{$n$}}
        \put(1,+26){\footnotesize{$n$}}
   \end{minipage} \hspace{10pt} \right)^{\otimes k}
   &=& (P_{n,i})^k
\hspace{1 mm}  
    \begin{minipage}[h]{0.09\linewidth}
        \vspace{0pt}
        \scalebox{0.36}{\includegraphics{horizantal_basis}}
             \put(-29,+17){\footnotesize{$2i$}}
        \put(-60,+26){\footnotesize{$n$}}
        \put(-60,-6){\footnotesize{$n$}}
        \put(1,-6){\footnotesize{$n$}}
        \put(1,+26){\footnotesize{$n$}}
   \end{minipage}
   \end{eqnarray*}
where 
\begin{eqnarray*}
P_{n,i}=\frac{ Tet\left[ 
\begin{array}{ccc}
2i & n & n \\ 
n & 2n & 2n%
\end{array}%
\right] Tet\left[ 
\begin{array}{ccc}
2i & 2n & 2n \\ 
n & n & n%
\end{array}%
\right]  }{\theta (2n,2n,2i)\theta (n,n,2i)}
\end{eqnarray*}

However, equation (\ref{22222}) implies:

\begin{eqnarray*}
P_{n,i}=\frac{ \theta(2n,2n,2i) }{\theta (n,n,2i)}
\end{eqnarray*}
Thus,

\begin{eqnarray*}
    \begin{minipage}[h]{0.34\linewidth}
        \vspace{0pt}
        \scalebox{0.36}{\includegraphics{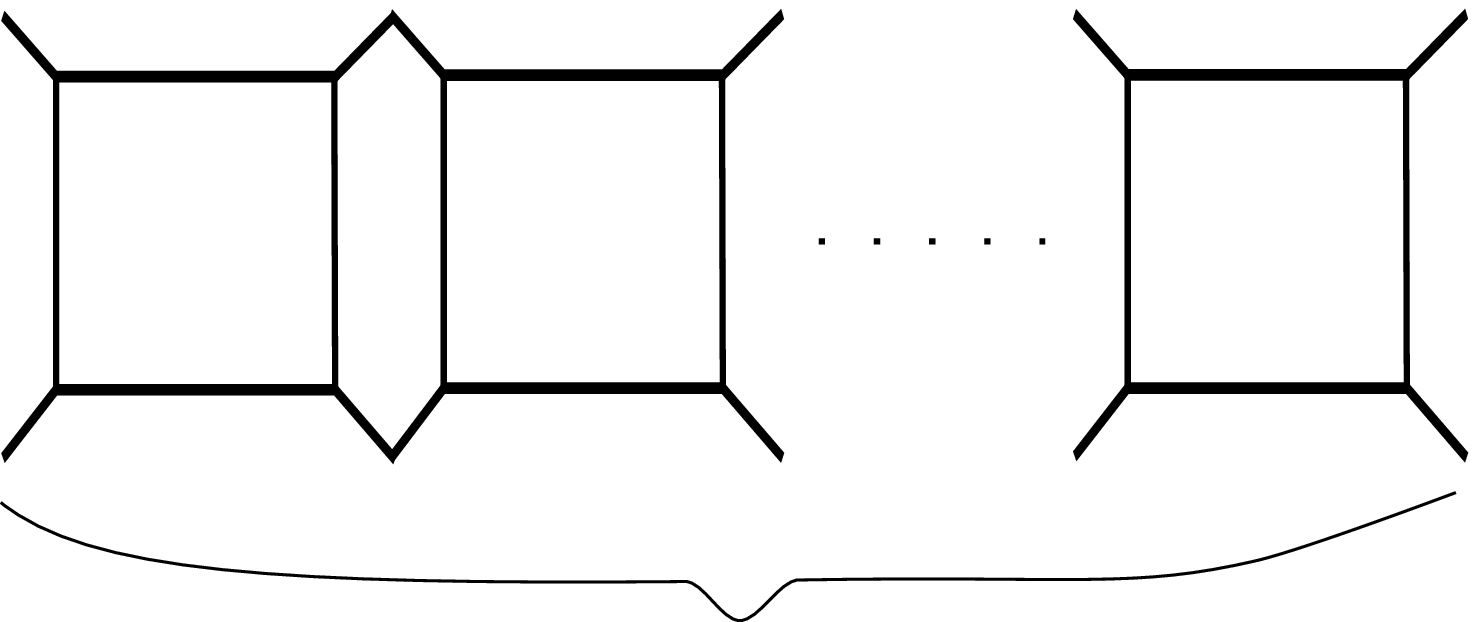}}
         \put(-29,+12){\footnotesize{$2n$}}
        \put(-29,+60){\footnotesize{$2n$}}
         \put(-40,64){\footnotesize{$n$}}
        \put(-60,-6){\footnotesize{$k$ copies}}
        \put(-5,+35){\footnotesize{$n$}}
        \put(-35,+35){\footnotesize{$n$}}
        \put(-45,+23){\footnotesize{$n$}}
         \put(-139,+12){\footnotesize{$2n$}}
        \put(-139,+60){\footnotesize{$2n$}}
         \put(-99,+60){\footnotesize{$2n$}}
        \put(-99,+12){\footnotesize{$2n$}}
        \put(-155,+29){\footnotesize{$n$}}
   \end{minipage}
   &=&\sum\limits_{i=0}^{n} B_{n,i}
\hspace{1 mm}  
   \begin{minipage}[h]{0.14\linewidth}
        \vspace{0pt}
        \scalebox{0.36}{\includegraphics{horizantal_basis}}
             \put(-29,+17){\footnotesize{$2i$}}
        \put(-60,+26){\footnotesize{$n$}}
        \put(-60,-6){\footnotesize{$n$}}
        \put(1,-6){\footnotesize{$n$}}
        \put(1,+26){\footnotesize{$n$}}
   \end{minipage}\otimes \hspace{5pt} \left( \hspace{10pt}
   \begin{minipage}[h]{0.11\linewidth}
        \vspace{0pt}
        \scalebox{0.22}{\includegraphics{bubble_as_spin}}
             \put(-29,+28){\footnotesize{$2n$}}
             \put(-29,-5){\footnotesize{$2n$}}
        \put(-60,+26){\footnotesize{$n$}}
        \put(-60,-6){\footnotesize{$n$}}
        \put(1,-6){\footnotesize{$n$}}
        \put(1,+26){\footnotesize{$n$}}
   \end{minipage} \hspace{10pt} \right)^{\otimes k-1}\\
   &=&\sum\limits_{i=0}^{n} B_{n,i} (P_{n,i})^{k-1}\hspace{1 mm}  
   \begin{minipage}[h]{0.14\linewidth}
        \vspace{0pt}
        \scalebox{0.36}{\includegraphics{horizantal_basis}}
             \put(-29,+17){\footnotesize{$2i$}}
        \put(-60,+26){\footnotesize{$n$}}
        \put(-60,-6){\footnotesize{$n$}}
        \put(1,-6){\footnotesize{$n$}}
        \put(1,+26){\footnotesize{$n$}}
   \end{minipage}.
   \end{eqnarray*}

The result follows.
\end{proof}

\begin{proposition}
\label{side1}
For $k\geq 1$ we have
\begin{equation*}
T_{L_k}(q)=(q;q)_{\infty}^{k+1} \sum_{i=0}^{\infty}\frac{q^{i}}{(q;q)_{i}^{k+1}}.
\end{equation*}

\begin{proof}
By theorem \ref{cody thm} we know that the tail of the graph $L_k$ is determined by the skein element $S^{(n)}_B(L_{k})$. This element is equivalent to the quantum spin network in Figure \ref{pretzel_quantum}. Note that there are $k+1$ copies of the box graph labeled by $2n$ and $n$ in Figure \ref{pretzel_quantum}.  

\begin{figure}[H]
  \centering
 {\includegraphics[scale=0.35]{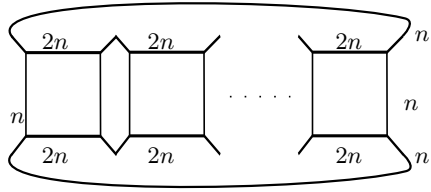}
 \put(-29,+53){\footnotesize{$2n$}}
        \put(-29,10){\footnotesize{$2n$}}
        \put(-100,10){\footnotesize{$2n$}}
        \put(-140,10){\footnotesize{$2n$}}
        \put(-100,+53){\footnotesize{$2n$}}
        \put(-140,53){\footnotesize{$2n$}}
        \put(-152,+25){\footnotesize{$n$}}
        \put(-3,30){\footnotesize{$n$}}
        \put(1,10){\footnotesize{$n$}}
        \put(1,+56){\footnotesize{$n$}}
  \caption{The quantum spin network corresponding to graph $L_k.$}
  \label{pretzel_quantum}}
\end{figure}
\end{proof}
By Lemma \ref{main lemma 1} we have
\begin{equation}
\label{graph evaluation}
S^{(n)}_B(L_{k}) = \sum_{i=0}^n \frac{\theta(2n,2n,2i)^{k+1}}{\theta(n,n,2i)^{k+1}}\Delta_{2i}
\end{equation}
However,
\begin{equation}
\label{special theta}
\theta(n,n,2i)=\frac{(-1)^{i+n} q^{- \frac{1}{2}(n+i)}
  (q, q)_i^2 (q, 
  q)_{-i +  n} (q, q)_{1 + i + n}}{(1 - q) (q, q)_{2 i} (q;q)_{ n}^2}.
\end{equation}
Putting (\ref{special theta}) in (\ref{graph evaluation}) and using Theorem \ref{cody thm} we obtain
\begin{equation}
T_{L_k} (q) \doteq_n S^{(n)}_B(L_{k})= \frac{1}{\Delta_n} \sum_{i=0}^n\left( \frac{(-1)^{-n}q^{n/2}(q;q)_{2n}^2(q;q)_{n-i}(q;q)_{n+i+1}}{(q;q)_{n}^2(q;q)_{2n-i}(q;q)_{2n+i+1}}\right)^{k+1}\Delta_{2i}
\end{equation}
Similar techniques to the ones used in Theorem \ref{technical} imply:
\begin{equation}
T_{L_{k}}(q) = (q;q)_{\infty}^{k+1} \sum_{i=0}^{\infty}  \frac{q^{i}}{(q;q)_i^{k+1}}
\end{equation}
\end{proposition}

\begin{proposition}
\label{side2}
For $k\geq 1$ we have
\begin{equation}
T_{L_k}= (q;q)_{\infty}^{k}\sum_{i_1=0}^{\infty}...\sum_{i_k=0}^{\infty}\frac{q^{\sum_{j=1}^k i_j+i_j^2+\sum_{s=2}^k \sum_{j=s}^k i_{s-1}i_j}}{\prod_{j=1}^k (q;q)_{i_j}  (q;q)_{\sum_{s=1}^j i_s} }
\end{equation}

\end{proposition}

\begin{proof}
We apply the bubble skein formula to obtain:
\begin{eqnarray*}
    \begin{minipage}[h]{0.25\linewidth}
        \vspace{0pt}
        \scalebox{0.07}{\includegraphics{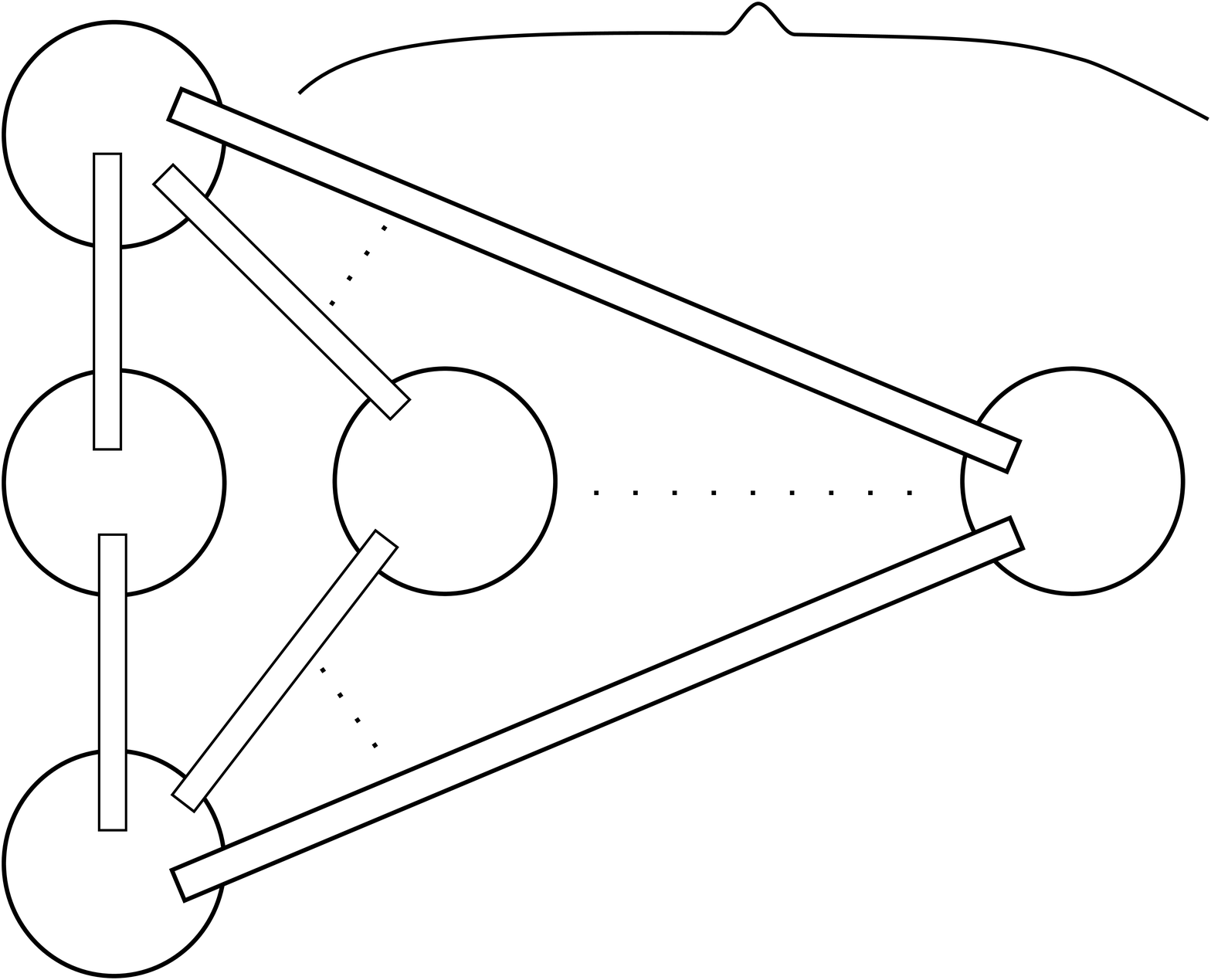}}
         \put(-30,64){\footnotesize{$n$}}
         \put(-132,59){\footnotesize{$n$}}
         \put(-130,21){\footnotesize{$n$}}
         \put(-70,62){\footnotesize{$n$}}
          \put(-100,44){\footnotesize{$n$}}
          \put(-130,88){\footnotesize{$n$}}
        \put(-60,102){\footnotesize{$k$ copies}}
   \end{minipage}
   &=&\sum\limits_{i=0}^{n}  \left\lceil 
\begin{array}{cc}
n & n \\ 
n & n%
\end{array}%
\right\rceil _{i_1}
\hspace{1 mm}  
   \begin{minipage}[h]{0.14\linewidth}
        \vspace{0pt}
        \scalebox{0.07}{\includegraphics{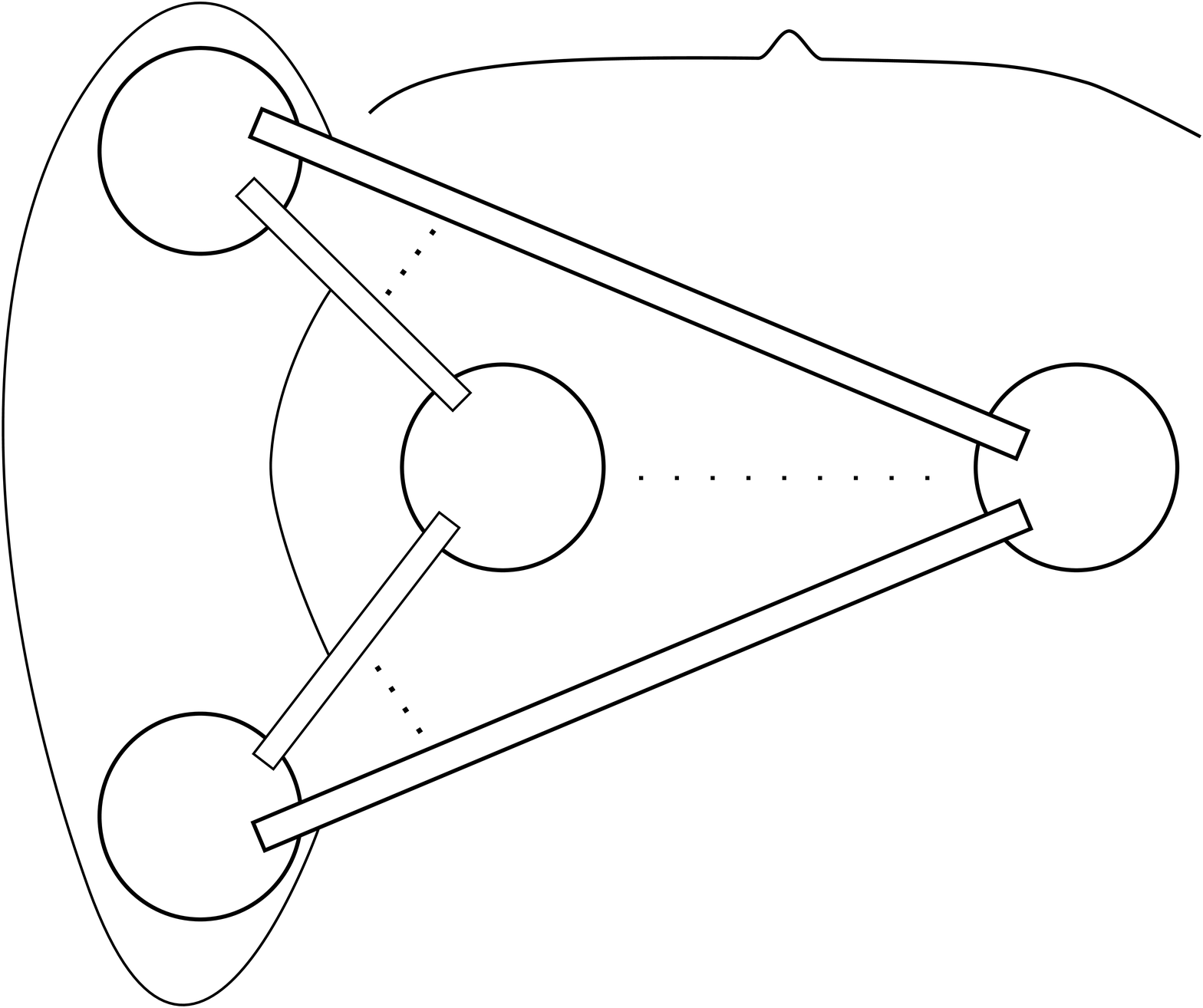}}
         \put(-17,40){\footnotesize{$n$}}
         \put(-132,59){\footnotesize{$i$}}
         \put(-130,33){\footnotesize{$n-i$}}
         \put(-65,65){\footnotesize{$n$}}
          \put(-102,55){\footnotesize{$i$}}
          \put(-130,78){\footnotesize{$n-i$}}
        \put(-60,115){\footnotesize{$k-1$ copies}}
   \end{minipage}\\&=&\sum\limits_{i=0}^{n}  \left\lceil 
\begin{array}{cc}
n & n \\ 
n & n%
\end{array}%
\right\rceil _{i_1}
\hspace{1 mm}  
  \begin{minipage}[h]{0.41\linewidth}
        \hspace{5.9pt}
        \scalebox{0.07}{\includegraphics{PRETZEL_SKEIN_1}}
         \put(-30,64){\footnotesize{$n+i$}}
         \put(-136,64){\footnotesize{$n+i$}}
         \put(-136,24){\footnotesize{$n-i$}}
         \put(-70,62){\footnotesize{$n$}}
          \put(-100,44){\footnotesize{$n$}}
          \put(-138,98){\footnotesize{$n-i$}}
        \put(-60,102){\footnotesize{$k-2$ copies}}
   \end{minipage}
   \end{eqnarray*}
The skein element in the last equation is obtained from the skein element in the first equation by isotopy of the strands and the properties of the Jones-Wenzl idempotent.
Similarly, we apply the bubble skein relation $(k-1)$ times on the skein element showing on the right handside of the previous equation to obtain

\begin{align}
\label{I hate you}
    \begin{minipage}[h]{0.28\linewidth}
        \vspace{0pt}
        \scalebox{0.07}{\includegraphics{PRETZEL_SKEIN_1}}
         \put(-30,64){\footnotesize{$n$}}
         \put(-132,59){\footnotesize{$n$}}
         \put(-130,21){\footnotesize{$n$}}
         \put(-70,62){\footnotesize{$n$}}
          \put(-100,44){\footnotesize{$n$}}
          \put(-130,88){\footnotesize{$n$}}
        \put(-60,102){\footnotesize{$k$ copies}}
   \end{minipage}
   &=\sum\limits_{i_1=0}^{n}\sum\limits_{i_2=0}^{n-i_1}...\sum\limits_{i_{k}=0}^{n-\sum_{l=1}^{k-1}i_l}  \left\lceil 
\begin{array}{cc}
n & n \\ 
n & n%
\end{array}%
\right\rceil _{i_1}\nonumber\\&\times\nonumber \prod_{j=1}^{k-1}\left\lceil 
\begin{array}{cc}
n-\sum_{s=1}^j i_s & n-\sum_{s=1}^j i_s \\ 
n+\sum_{s=1}^j i_s & n%
\end{array}%
\right\rceil _{i_{j+1}}
\hspace{1 mm}
   \begin{minipage}[h]{0.14\linewidth}
        \vspace{0pt}
        \scalebox{0.07}{\includegraphics{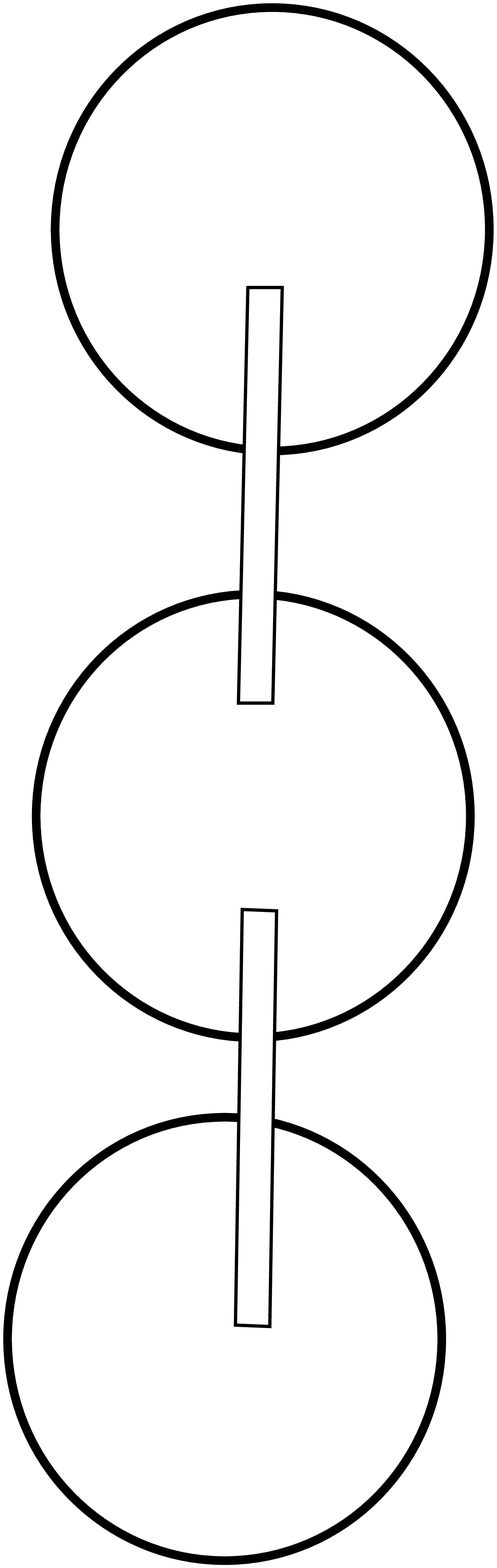}}
    \put(-30,86){\footnotesize{$n-\sum_{s=1}^{k-1} i_s$}}
     \put(-30,-16){\footnotesize{$n-\sum_{s=1}^{k-1} i_s$}}
     \put(-2,46){\footnotesize{$n+\sum_{s=1}^{k-1} i_s$}}    
   \end{minipage}\\&=\sum\limits_{i_1=0}^{n}\sum\limits_{i_2=0}^{n-i_1}...\sum\limits_{i_{k}=0}^{n-\sum_{l=1}^{k-1}i_l}  \left\lceil 
\begin{array}{cc}
n & n \\ 
n & n%
\end{array}%
\right\rceil _{i_1}\nonumber\\&\times \prod_{j=1}^{k-1}\left\lceil 
\begin{array}{cc}
n-\sum_{s=1}^j i_s & n-\sum_{s=1}^j i_s \\ 
n+\sum_{s=1}^j i_s & n%
\end{array}%
\right\rceil _{i_{j+1}}\frac{(\Delta_{2n})^2}{\Delta_{n+\sum_{s=1}^{k-1} i_s}}
   \end{align}

However,

\begin{equation}
\label{bubble special}
\left\lceil 
\begin{array}{cc}
n-a & n-a\\ 
n+a & n%
\end{array}%
\right\rceil _{i}=  \frac{(-1)^{i + n} q^{i/2 + a i + i^2 - n/2} (q, q)_{n}^3 (q, q)_{-a + n}^2 (q, q)_{
   a + n} (q, q)_{1 - a - i + 3 n}}{
(q, q)_i (q, q)_{a + i} (q, q)_{ 
  2 n}^2 (q, q)_{-i + n} (q, 
  q)_{-a - i + n}^2 (q, q)_{1 - a + 2 n}}
\end{equation}
Using equation (\ref{bubble special}) in \ref{I hate you} and using similar techniques to ones we used in Theorem \ref{technical} we obtain the result.
\end{proof}
Propositions \ref{side1} and \ref{side2} imply immediately the following
\begin{corollary}
For $k \geq 1$ we have
\begin{equation*}
(q;q)_{\infty} \sum_{i=0}^{\infty}  \frac{q^{i}}{(q;q)_i^{k+1}} = \sum_{i_1=0}^{\infty}...\sum_{i_k=0}^{\infty}\frac{q^{\sum_{j=1}^k i_j+i_j^2+\sum_{s=2}^k \sum_{j=s}^k i_{s-1}i_j}}{\prod_{j=1}^k (q;q)_{i_j}  (q;q)_{\sum_{s=1}^j i_s} }
\end{equation*}
\end{corollary}


\begin{thebibliography}{15}


\bib{Andrews}{book}{
   author={Andrews, George E.},
   author={Berndt, Bruce C.},
   title={Ramanujan's lost notebook. Part I},
   publisher={Springer, New York},
   date={2005},
   pages={xiv+437},
   isbn={978-0387-25529-3},
   isbn={0-387-25529-X},
   review={\MR{2135178 }},
}


\bib{Armond}{article}{
   author={Armond, Cody W.},
   title={The head and tail conjecture for alternating knots},
   journal={Algebr. Geom. Topol.},
   volume={13},
   date={2013},
   number={5},
   pages={2809--2826},
   issn={1472-2747},
   review={\MR{3116304}},
}

\bib{Armond1}{article}{
   author={Armond, Cody W.},
   title={Walks along braids and the colored Jones polynomial},
   journal={J. Knot Theory Ramifications},
   volume={23},
   date={2014},
   number={2},
   pages={1450007, 15},
   issn={0218-2165},
   review={\MR{3197051}},
}

\bib{CodyOliver}{article}{
   author={Armond, Cody W.},
   author={Dasbach,  Oliver T.},
   title={Rogers-Ramanujan type identities and the head and tail of the colored Jones polynomial},
   journal={arXiv:1106.3948},
   volume={},
   date={2011},
   number={},
   pages={},
}

\bib{BHMV92}{article}{
   author={Blanchet, C.},
   author={Habegger, N.},
   author={Masbaum, G.},
   author={Vogel, P.},
   title={Three-manifold invariants derived from the Kauffman bracket},
   journal={Topology},
   volume={31},
   date={1992},
   number={4},
   pages={685--699},
   issn={0040-9383},
   review={\MR{1191373 }},
}
	

\bib{DL}{article}{
   author={Dasbach, Oliver T.},
   author={Lin, Xiao-Song},
   title={On the head and the tail of the colored Jones polynomial},
   journal={Compos. Math.},
   volume={142},
   date={2006},
   number={5},
   pages={1332--1342},
   issn={0010-437X},
   review={\MR{2264669 }},
}

\bib{DL1}{article}{
   author={Dasbach, Oliver T.},
   author={Lin, Xiao-Song},
   title={A volumish theorem for the Jones polynomial of alternating knots},
   journal={Pacific J. Math.},
   volume={231},
   date={2007},
   number={2},
   pages={279--291},
   issn={0030-8730},
   review={\MR{2346497 }},
}


\bib{HOMFLY}{article}{
   author={Freyd, P.},
   author={Yetter, D.},
   author={Hoste, J.},
   author={Lickorish, W. B. R.},
   author={Millett, K.},
   author={Ocneanu, A.},
   title={A new polynomial invariant of knots and links},
   journal={Bull. Amer. Math. Soc. (N.S.)},
   volume={12},
   date={1985},
   number={2},
   pages={239--246},
   issn={0273-0979},
   review={\MR{776477 }},
}

\bib{GL}{article}{
   author={Garoufalidis, Stavros},
   author={L{\^e}, Thang T. Q.},
   title={Nahm sums, stability and the colored Jones polynomial},
   journal={Res. Math. Sci.},
   volume={2},
   date={2015},
   pages={2:1},
   issn={2197-9847},
   review={\MR{3375651}},
}

\bib{Hajij3}{article}{
   author={Hajij, Mustafa},
   title={The colored kauffman skein relation and the head and tail of the colored jones polynomial},
   journal={arXiv preprint arXiv:1401.4537},
   volume={},
   date={2014},
   number={},
   pages={},
   issn={},
   review={},
}



\bib{Hajij2}{article}{
   author={Hajij, Mustafa},
   title={The tail of a quantum spin network},
   journal={Ramanujan J.  },
   volume={Vol ?},
   date={2015},
   number={},
   pages={pp 1-42},
   issn={},
   review={},
}

\bib{Hajij1}{article}{
   author={Hajij, Mustafa},
   title={The Bubble skein element and applications},
   journal={J. Knot Theory Ramifications},
   volume={23},
   date={2014},
   number={14},
   pages={1450076, 30},
   issn={0218-2165},
   review={\MR{3312619}},
}
	
	
\bib{J2}{article}{
   author={Jones, V. F. R.},
   title={Hecke algebra representations of braid groups and link
   polynomials},
   journal={Ann. of Math. (2)},
   volume={126},
   date={1987},
   number={2},
   pages={335--388},
   issn={0003-486X},
   review={\MR{908150 }},
}

\bib{J1}{article}{
	author={Jones, V. F. R.},
	title={A polynomial invariant for knots via von Neumann algebras},
	journal={Bull. Amer. Math. Soc. (N.S.)},
	volume={12},
	date={1985},
	number={1},
	pages={103--111},
	issn={0273-0979},
	review={\MR{766964 }},
}

\bib{Hikami}{article}{
   author={Hikami, Kazuhiro},
   title={Volume conjecture and asymptotic expansion of $q$-series},
   journal={Experiment. Math.},
   volume={12},
   date={2003},
   number={3},
   pages={319--337},
   issn={1058-6458},
   review={\MR{2034396 }},
}

\bib{Kauffman}{article}{
   author={Kauffman, Louis H.},
   title={State models and the Jones polynomial},
   journal={Topology},
   volume={26},
   date={1987},
   number={3},
   pages={395--407},
   issn={0040-9383},
   review={\MR{899057 }},
}


\bib{Robert}{article}{
  author={Keilthy, A.},
   author={Osburn, R.},
   title={Rogers-Ramanujan identities for alternating knots},
   journal={Journal of Number Theory},
   date={2015},
}


\bib{Lic92}{article}{
   author={Lickorish, W. B. R.},
   title={Calculations with the Temperley-Lieb algebra},
   journal={Comment. Math. Helv.},
   volume={67},
   date={1992},
   number={4},
   pages={571--591},
   issn={0010-2571},
   review={\MR{1185809 }},
}

\bib{MV}{article}{
   author={Masbaum, G.},
   author={Vogel, P.},
   title={$3$-valent graphs and the Kauffman bracket},
   journal={Pacific J. Math.},
   volume={164},
   date={1994},
   number={2},
   pages={361--381},
   issn={0030-8730},
   review={\MR{1272656 }},
}
	
\bib{Morton}{article}{
	author={Morton, H. R.},
	title={Invariants of links and $3$-manifolds from skein theory and from
		quantum groups},
	conference={
		title={Topics in knot theory},
		address={Erzurum},
		date={1992},
	},
	book={
		series={NATO Adv. Sci. Inst. Ser. C Math. Phys. Sci.},
		volume={399},
		publisher={Kluwer Acad. Publ., Dordrecht},
	},
	date={1993},
	pages={107--155},
	review={\MR{1257908}},
}	

\bib{Murakami}{article}{
	author={Murakami, Hitoshi},
	author={Murakami, Jun},
	title={The colored Jones polynomials and the simplicial volume of a knot},
	journal={Acta Math.},
	volume={186},
	date={2001},
	number={1},
	pages={85--104},
	issn={0001-5962},
	review={\MR{1828373 }},
}	
		
\bib{Przytycki}{article}{
   author={Przytycki, J{\'o}zef H.},
   title={Fundamentals of Kauffman bracket skein modules},
   journal={Kobe J. Math.},
   volume={16},
   date={1999},
   number={1},
   pages={45--66},
   issn={0289-9051},
   review={\MR{1723531 }},
}
\bib{PT}{article}{
   author={Przytycki, J{\'o}zef H.},
   author={Traczyk, Pawe{\l}},
   title={Invariants of links of Conway type},
   journal={Kobe J. Math.},
   volume={4},
   date={1988},
   number={2},
   pages={115--139},
   issn={0289-9051},
   review={\MR{945888 }},
}



\bib{RT}{article}{
   author={Reshetikhin, N.},
   author={Turaev, V. G.},
   title={Invariants of $3$-manifolds via link polynomials and quantum
   groups},
   journal={Invent. Math.},
   volume={103},
   date={1991},
   number={3},
   pages={547--597},
   issn={0020-9910},
   review={\MR{1091619 }},
}

\bib{TuraevViro92}{article}{
   author={Turaev, V. G.},
   author={Viro, O. Ya.},
   title={State sum invariants of $3$-manifolds and quantum $6j$-symbols},
   journal={Topology},
   volume={31},
   date={1992},
   number={4},
   pages={865--902},
   issn={0040-9383},
   review={\MR{1191386 }},
}

\bib{TuraevWenzl93}{article}{
   author={Turaev, V.},
   author={Wenzl, H.},
   title={Quantum invariants of $3$-manifolds associated with classical
   simple Lie algebras},
   journal={Internat. J. Math.},
   volume={4},
   date={1993},
   number={2},
   pages={323--358},
   issn={0129-167X},
   review={\MR{1217386 }},
}


\bib{Ramanujan}{book}{
   author={Ramanujan, Srinivasa},
   title={The lost notebook and other unpublished papers},
   note={With an introduction by George E. Andrews},
   publisher={Springer-Verlag, Berlin; Narosa Publishing House, New Delhi},
   date={1988},
   pages={xxviii+419},
   isbn={3-540-18726-X},
   review={\MR{947735}},
}

\bib{Wenzl}{article}{
   author={Wenzl, Hans},
   title={On sequences of projections},
   journal={C. R. Math. Rep. Acad. Sci. Canada},
   volume={9},
   date={1987},
   number={1},
   pages={5--9},
   issn={0706-1994},
   review={\MR{873400 }},
}	
		
		
	

\end{thebibliography}
\end{document}